\documentclass[12pt]{amsart}
\usepackage{amsfonts,amssymb,amsthm,amsmath,graphicx, hyperref}
\newtheorem{thm}{Theorem}[section]
\newtheorem{cor}[thm]{Corollary}
\newtheorem{lemma}[thm]{Lemma}
\newtheorem{proposition}[thm]{Proposition}

\newcommand{\Z}{\mathbb{Z}}
\newcommand{\SL}{{\text {\rm SL}}}
\newcommand{\abs}[1]{\left|#1\right|}
\newcommand{\sm}[4]{\left(\begin{smallmatrix} #1 & #2 \\ #3 & #4\end{smallmatrix}\right)}

\begin{document}

\title{Zeros of modular forms of half integral weight}
\author{Amanda Folsom}
\author{Paul Jenkins}
\subjclass[2010]{11F37, 11F30}

\begin{abstract}
We study canonical bases for spaces of weakly holomorphic modular forms of level 4 and weights in $\mathbb{Z}+\frac{1}{2}$ and show that almost all modular forms in these bases have the property that many of their zeros in a fundamental domain for $\Gamma_0(4)$ lie on a lower boundary arc of the fundamental domain.  Additionally, we show that at many places on this arc, the generating function for Hurwitz class numbers is equal to a particular mock modular Poincar\'{e} series, and show that for positive weights, a particular set of Fourier coefficients of cusp forms in this canonical basis cannot simultaneously vanish.
\end{abstract}
\maketitle

\section{Introduction}
In studying functions of a complex variable, a natural problem is to determine the locations of the zeros of the functions.  There are a number of recent interesting results on the zeros of modular forms.  For instance, the zeros of Hecke eigenforms of integer weight $k$ become equidistributed in the fundamental domain as $k \rightarrow \infty$ (see~\cite{HS} and~\cite{Ru}), yet such forms still have many zeros on the boundary and center line of the fundamental domain~\cite{GS, LMR}.

Duke and the second author~\cite{DJ1} studied zero locations for a canonical basis $\{f_{k, m}(\tau)\}$ for spaces of integer weight weakly holomorphic modular forms for $\SL_2(\Z)$.  If the weight $k \in 2\Z$ is written as $12\ell + k'$ with $k' \in \{0, 4, 6, 8, 10, 14\}$, then the basis elements have Fourier expansions of the form $f_{k, m}(\tau) = q^{-m} + O(q^{\ell+1})$, where $q = e^{2\pi i \tau}$ as usual.  If $m \geq |\ell|-\ell$, then all of the zeros of $f_{k, m}(\tau)$ in the standard fundamental domain for $\SL_2(\Z)$ lie on the unit circle.  These results were extended~\cite{GJ, HJ} to similar canonical bases $\{f_{k, m}^{(N)}(\tau)\}$ for the spaces $M_k^{\sharp}(N)$ of weakly holomorphic modular forms of integer weight $k$ and level $N = 2, 3, 4$ with poles only at the cusp at $\infty$, showing that many of the zeros of the basis element $f_{k, m}^{(N)}(\tau)$ lie on an appropriate arc if $m$ is large enough.

For spaces of modular forms with weight $k \in \Z+\frac{1}{2}$, canonical bases with similar Fourier expansions exist.  Zagier~\cite{Z} defined such bases for spaces of weakly holomorphic modular forms of level $4$ and weights $\frac{1}{2}$ and $\frac{3}{2}$ satisfying Kohnen's plus space condition, and analogous bases were shown in~\cite{DJ2} to exist in level $4$ for all weights $k \in \Z+\frac{1}{2}$.  In this paper, we address the natural question of whether the zeros of the modular forms in such a basis lie on an arc.

Let $\vartheta(\tau)$ be the theta function $\vartheta(\tau) = \sum_{n\in \mathbb Z} q^{n^2}$.  For $k = s + \frac{1}{2}$ with $s \in \Z$, let $M_k^!$ denote the space of holomorphic functions on the upper half $\mathcal{H}$ of the complex plane which may have poles at the cusps, which transform like $\vartheta^{2k}$ under the action of the group $\Gamma_0(4)$,
and which satisfy Kohnen's plus space condition, so that their Fourier expansion is of the form $\sum_{(-1)^s n \equiv 0, 1 \pmod{4}} a(n) q^n$.  As shown in \cite{DJ2}, the space $M_k^!$ has a canonical basis of modular forms $\{f_{k, m}(\tau)\}$ with Fourier expansion given by \[f_{k, m}(\tau) = q^{-m} + O(q^{N+1}),\] where $N$ depends only on $k$.  Decompose the integer $s$ as $s = 12a + b$, where $a \in \Z$ and \[b \in \{6, 8, 9, 10, 11, 12, 13, 14, 15, 16, 17, 19\}.\]  We show in section~\ref{sec2} that each of these basis elements $f_{k, m}(\tau)$ has exactly $5a + \frac{5m}{4} + \frac{b}{2} - C - \epsilon$ zeros in a fundamental domain for $\Gamma_0(4)$, where $C \in \{\frac{3}{4}, 1, \frac{3}{2}, \frac{7}{4}\}$ depends only on $b$ and the parity of $m$, and where the nonnegative integer $\epsilon$, defined precisely in section~\ref{sec2}, counts the number of initial zeros in a certain sequence of Fourier coefficients of $f_{k, m}$. 

Let $\mathcal{A}$ be the arc in the upper half plane given by \[\mathcal{A} = \left\{ \left. -\frac{1}{4} + \frac{1}{4}e^{i\theta} \  \right| \ 0 < \theta < \pi\right\}.\]
In this paper, we will prove the following theorem about the zeros of the basis elements $f_{k, m}(\tau)$.
\begin{thm}\label{mainthm}
For $k \in \Z+\frac{1}{2}$, assume the notation above.  There are absolute, positive
constants $A$ and $B$ such that if $m \geq A|a| + B$, then at least
$2a + \frac{m}{2}$ of the $5a + \frac{5m}{4} + \frac{b}{2} - C - \epsilon$ zeros of $f_{k, m}$ in a fundamental domain for $\Gamma_0(4)$ lie on the arc $\mathcal{A}$.  Additionally, $A \leq 9$ and $B \leq 109$.
\end{thm}
The proof of the theorem uses contour integration on a generating function for the canonical basis elements $f_{k, m}$ to approximate $f_{k, m}$ by a real-valued trigonometric function on $\mathcal{A}$.  We note that for specific values of $k$, the bounds on $A$ and $B$ and the quantity $2a+\frac{m}{2}$ can often be improved; for instance, if $a \geq 0$ we may take $A = 0$.

In the case of weight $k=\frac{3}{2}$, Theorem~\ref{mainthm} yields an interesting corollary pertaining to weak Maass forms and mock modular forms.  Loosely speaking, a weak Maass form is a smooth complex function defined on the upper half of the complex plane which transforms like a modular form, but is not necessarily holomorphic.  Such a function must also be annihilated by a Laplacian operator, and satisfy suitable growth conditions in the cusps.  From the definition, it follows that a weak Maass form naturally decomposes into two parts, a ``holomorphic part,'' and a ``non-holomorphic'' part;  the holomorphic parts are commonly referred to as mock modular forms.  (See \cite{BF} or \cite{OnoCDM}, for example.)  A first example of a half-integral weight weak Maass form arises from the function \[\vartheta(\tau)^3 = 1 + 6\tau + 12 q^2 + 8 q^3 + 6 q^4 + 24 q^5 + 24 q^6 + 12 q^8 + \cdots.\]
If we define its Fourier coefficients by $\vartheta(\tau)^3 =: \sum_{n\geq 0} h(n) q^n$, then it was proved by Gauss that \[h(n) = \begin{cases} 12 H(4n), & n\equiv 1,2 \pmod{4}, \\ 24 H(n), & n\equiv 3 \pmod{8}, \\ h\left(\frac{n}{4}\right), & n\equiv 0 \pmod{4}, \\ 0, & n\equiv 7 \pmod{8}, \end{cases}\] where the values $H(n)$ are the Hurwitz class numbers.
Let $\mathcal G(\tau) := -\frac{1}{12} + \sum_{n\geq 0} H(n) q^n$ be the generating function for the Hurwitz class numbers.  Zagier \cite{ZagierEis} beautifully showed that $\mathcal G $ is in fact the holomorphic part of a  weak Maass form.    More precisely, Zagier showed that the function
\[G(\tau) := \mathcal G(\tau) + \frac{1}{8\sqrt{\pi}} \sum_{n\in\mathbb Z} n \Gamma\left(-\frac12,4\pi n^2 \text{Im}(\tau)\right)q^{-n^2}\] is a weight $\frac{3}{2}$ weak Maass form of moderate growth on $\Gamma_0(4)$.  Here, $\Gamma(\cdot,\cdot)$ denotes the incomplete gamma function.

By work of Bruinier, the second author, and Ono~\cite{BJO}, it turns out that Zagier's weak Maass form $G$ is naturally related to the basis functions $f_{\frac{3}{2},m}$ studied here, for any positive integer $m\equiv 0,1 \pmod{4}$ which is a square.  Precisely, for such $m$, the authors show that
\[f_{\frac{3}{2},m} = 24 G + F_{-m,\frac32}^+ ,\]
where the functions \[F_{-m,k}^+(\tau) := \tfrac32 F_{-m}(\tau,\tfrac{k}{2}) |_k \ \text{pr}\] are the weight $k$ and index $m$ Poincar\'e series defined in \cite[(3.13)]{BJO}, which are weight $k$ weak Maass forms of level $4$.
Denote the holomorphic part of $F_{-m,\frac32}^+ $ by $ \mathcal{F}_{-m,\frac32}^+$.  Because $f_{\frac{3}{2}, m}$ is weakly holomorphic, we must have that \[f_{\frac{3}{2},m} = 24 \mathcal G + \mathcal{F}_{-m,\frac32}^+ .\] That is, the basis elements $f_{\frac{3}{2},m} $ may be expressed as the sum of two mock modular forms, one of which is the generating function $24\mathcal G $ for Hurwitz class numbers, and the other of which is the holomorphic part of the Poincar\'e series $F_{-m,\frac32}^+ $.  By Theorem~\ref{mainthm},  the basis functions $f_{\frac32, m} $ have many zeros on $\mathcal{A}$;  thus, it must be the case that the mock modular forms given by $24\mathcal G $ and $\mathcal{F}_{-m,\frac32}^+ $ take on equal and opposite values at many points on $\mathcal{A}$.  Precisely, we have the following theorem.

\begin{thm}
Given the above conditions on $m$, there are at least
$\frac{m}{2} - 1$ points $\tau_j^{(m)}$ on the arc $\mathcal{A}$ at which the mock modular generating function for Hurwitz class numbers $\mathcal G$ and the mock modular Poincar\'e series $\frac{-1}{24}\mathcal{F}_{-m,\frac32}^+$ take on equal values.  That is, for such points $\tau_j^{(m)}$, we have that \[\mathcal G(\tau_j^{(m)}) = \tfrac{-1}{24} \mathcal{F}_{-m,\frac32}^+(\tau_j^{(m)}).\]
\end{thm}
We note that a direct application of Theorem~\ref{mainthm} gives $\frac{m}{2} -2$ zeros of the form $f_{\frac{3}{2}, m}$; however, the proof of Theorem~\ref{mainthm} in section~\ref{sec3} shows that for the weight $k=\frac{3}{2}$, with $b=13$, there is at least one extra zero.

Our results also lead to an interesting corollary in the case of weight $\frac32$ for any positive integer $m\equiv 0,1 \pmod{4}$ which is not a square.  In this case, it turns out that Theorem~\ref{mainthm} provides information about the zeros of the Poincar\'e series $F_{-m,\frac32}^+$ themselves.  Previously, Rankin~\cite{Ra1} addressed the problem of understanding the number of zeros of general Poincar\'e series of level 1 and even integer weight at least $4$, and gave an explicit bound on the number of zeros lying on the intersection of the standard fundamental domain with the boundary of the unit disk. Our work leads to  results analogous to those of Rankin in the case of weight $\frac32$, by virtue of the fact that for positive integers $m\equiv 0,1\pmod{4}$ such that $m$ is not a square, we have from~\cite{BJO} that $f_{\frac{3}{2},m} =  F_{-m,\frac32}^+ $.  Thus, we have the following Corollary to Theorem~\ref{mainthm}, giving locations for the zeros of certain weight $\frac{3}{2}$ Poincar\'e series.  The number $C_m$ is equal to $\frac{3}{2}$ or $\frac{3}{4}$, depending on whether $m$ is even or odd.

\begin{cor} \label{corollary13} Assume the notation above, and let $m\equiv 0,1 \pmod{4}$ such that $m$ is not a square.   There is an absolute  positive constant $A$ such that if $m \geq A$, then at least
$\frac{m}{2}$ of the $\frac{5m}{4}+\frac32  - C_m - \epsilon$ zeros of the Poincar\'e series $F_{-m,\frac32}^+ $ in a fundamental domain for $\Gamma_0(4)$ lie on the arc $\mathcal{A}$.  Additionally, $A \leq 111$.
\end{cor}
We note again that the bounds appearing here, for $k = \frac{3}{2}$, are better than appear in the general statement of Theorem~\ref{mainthm}.  Details supporting this computation appear in section~\ref{secerrorbound}.

As an additional application, for any weight $k \in \Z+\frac{1}{2}$, Theorem~\ref{mainthm} implies the following non-vanishing theorem for coefficients of modular forms in the canonical basis for $M_{2-k}^!$.
\begin{thm}\label{lehmerthm}
Assume the notation from Theorem~\ref{mainthm}.
Let $m \geq A\abs{a}+B$ with $m \equiv 0, (-1)^{s-1} \pmod{4},$ and let $\{f_{2-k, i}(\tau)\}$ be the canonical basis for the space $M_{2-k}^!$.  For any integer
$M > 3a + \frac{3m}{4} + \frac{b}{2} - C$, it is impossible for the Fourier coefficients of $q^m$ in each of the first $M$ basis elements $f_{2-k, i}(\tau)$ with $i \not\equiv m \pmod{4}$ to simultaneously vanish.
\end{thm}
If $2-k$ is positive and large enough, many of the basis elements in Theorem~\ref{lehmerthm} are actually cusp forms.  An analogous result in integer weights would resemble a weaker form of Lehmer's conjecture on the nonvanishing of the coefficients of the cusp form $\Delta(\tau)$, since $\Delta$ is the first canonical basis element $f_{12, -1}$ of weight 12.  See also~\cite{ChenWu} for additional results on the nonvanishing of coefficients of cusp forms of half integral weight.

In section~\ref{sec2} of this paper, we give definitions, notation, and the proof of Theorem~\ref{lehmerthm}.  The main argument in the proof of Theorem~\ref{mainthm} appears in section~\ref{sec3}, with associated computations appearing in sections~\ref{sec_akm},~\ref{pfsec2}, and~\ref{secerrorbound}.

\section{Definitions}\label{sec2}

As before, let $k = s + \frac{1}{2}$ be half integral, and write $s = 12a + b = 6\ell + \frac{k'}{2}$ with $a, \ell \in \Z,$ and \[b \in \{6, 8, 9, 10, 11, 12, 13, 14, 15, 16, 17, 19\}\] and $k' \in \{0, 4, 6, 8, 10, 14\}$.  The basis elements $f_{k, m}(\tau)$ for the space $M_k^!$ of weakly holomorphic modular forms of weight $k$ and level $4$ satisfying Kohnen's plus space condition are constructed explicitly in~\cite{DJ2} in terms of
the weight $2$, level $4$ Eisenstein series given by \[F(\tau) = \sum_{n=0}^\infty \sigma(2n+1)q^{2n+1},\] the weight $12$, level $1$ cusp form $\Delta(\tau)$, and
the theta function $\vartheta(\tau) = \sum_{n \in \Z} q^{n^2} \in M_{\frac{1}{2}}^!$, which transforms under $\gamma = \sm{a}{b}{c}{d} \in \Gamma_0(4)$ with character $ \rho = \rho_\gamma := \left(\frac{c}{d}\right)\varepsilon_d^{-1} ,$ where $\varepsilon_d$ is equal to $1$ or $i$, depending on whether $d$ is congruent to $1$ or $3$ $\pmod 4$.
The basis elements have a Fourier expansion given by \[f_{k, m}(\tau) = q^{-m} + \mathop{\sum_{n > N}}_{(-1)^s n \equiv 0, 1\pmod{4}}a_k(m, n) q^n.\] Here $N$ is equal to $2\ell$ if $\ell$ is even and $2\ell - (-1)^s$ if $\ell$ is odd, and $m \geq -N$ satisfies $(-1)^{s-1}m \equiv 0, 1\pmod{4}$.  All of the coefficients $a_k(m, n)$ are integers.

Any fundamental domain for $\Gamma_0(4)$ has three cusps, which we take to be at $\infty, 0$, and $\frac{1}{2}$ and which have widths $1$, $4$, $1$ respectively.  We use a fundamental domain $\mathcal{F}$ bordered by the vertical lines with real part $\pm \frac{1}{2}$ and the semicircles defined by $\pm \frac{1}{4} + \frac{1}{4}e^{i\theta}$ for $\theta \in [0, \pi]$.

\begin{figure}[!ht]
\centering
\includegraphics[width = .5\textwidth]{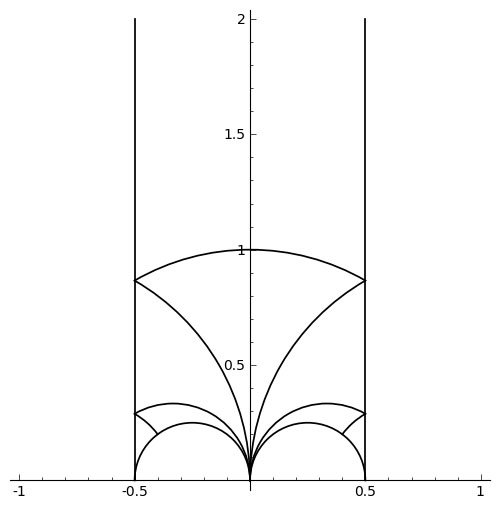}
\caption{Fundamental domain $\mathcal F$ for $\Gamma_0(4)$}
\end{figure}

The valence formula for $\Gamma_0(4)$ (see, for instance,~\cite{Rankin}) tells us that for a modular form $f$ of half integral weight $k$ for $\Gamma_0(4)$, we have \[\sum_{\tau \in \mathcal{F}} \textrm{Ord}_{\Gamma_0(4)}(f, \tau) + \sum_{\zeta \in \{\infty, 0, \frac{1}{2}\}} \textrm{Ord}_{\Gamma_0(4)}(f, \zeta) = \frac{k}{2}.\]  Here, for points $\tau$ in the upper half plane, $\textrm{Ord}_{\Gamma_0(4)}(f, \tau) = \textrm{ord}_{\Gamma_0(4)}(f, \tau)$, the usual order of $f$ as an analytic function at $\tau$.  For the cusps $\zeta \in\{ \infty, 0, \frac{1}{2}\}$,
the quantity $\textrm{Ord}_{\Gamma_0(4)}(f, \zeta)$ is determined by the $q$-expansion at each cusp.  Specifically, $\textrm{Ord}_{\Gamma_0(4)}(f, \infty)$ is the smallest exponent in the $q$-expansion of $f$ at $\infty$, while for the cusp at $0$ we multiply the smallest power of $q$ in the Fourier expansion of $f$ at $0$ by the cusp width 4 to get $\textrm{Ord}_{\Gamma_0(4)}(f, 0)$.  For the cusp at $\frac{1}{2}$ the cusp width is $1$, and we find that $\textrm{Ord}_{\Gamma_0(4)}(f, \frac{1}{2})$ is again the smallest exponent in the $q$-expansion at the cusp, but we note that since $\frac{1}{2}$ is an irregular cusp for $\Gamma_0(4)$ for weights in $\Z+\frac{1}{2}$, this exponent will be in $\Z+\frac{1}{4}$ if $s$ is odd and in $\Z+\frac{3}{4}$ if $s$ is even.  For example, we find that for $\vartheta(\tau) \in M_{\frac{1}{2}}^!$, we have
\[\begin{array}{lcl} \textrm{Ord}_{\Gamma_0(4)}(\vartheta, \infty) &=& 0,  \\
 \textrm{Ord}_{\Gamma_0(4)}(\vartheta, 0) &=& 0,  \\
 \textrm{Ord}_{\Gamma_0(4)}\left(\vartheta, \frac{1}{2}\right) &=& \frac{1}{4}. \end{array}\]

We next want to understand the behavior of the basis element $f_{k, m}(\tau)$ at the cusps at $0$ and $\frac{1}{2}$.  To do so, we use the $N=4$ case of a theorem appearing in~\cite{GreenJenkins} for levels $12, 20, 28, 52$.
Performing explicit computations with the projection operator similar to those in~\cite{Kohnen}, we find that if $g(\tau) \in M_k^!(4)$, then the expansions of $g$ at the cusps at $0$ and $\frac{1}{2}$ are given by
\begin{align*}g^{(0)}(\tau) &= C_1\left(g\left(\frac{\tau-1}{16}\right)+g\left(\frac{\tau+7}{16}\right)\right),\\
g^{(\frac{1}{2})}(\tau) &= C_2\left(g\left(\frac{2\tau+1}{8}\right) - g\left(\frac{2\tau-3}{8}\right)\right),\end{align*} where $C_1$ and $C_2$ are constants.  Thus, if the Fourier expansion of $g(\tau)$ is given by $g(\tau) = \sum a(n) q^n$, we have
\[g^{(0)}(\tau) = C_1 \sum_n \left( a(n) e^{2\pi i n (\frac{\tau-1}{16})} + a(n) e^{2\pi i n (\frac{\tau+7}{16})}\right),\]
and since $\frac{\tau-1}{16}$ and $\frac{\tau+7}{16}$ differ by exactly $\frac{1}{2}$, then the two terms cancel when $n$ is odd.  Since $n \equiv 0, (-1)^s \pmod{4}$, this gives
\[g^{(0)}(\tau) = 2 C_1 \sum_n a(4n) e^{2\pi i n (\frac{\tau-1}{4})}.\]  Similarly, we get cancellation for even exponents $n$ in the expansion of $g^{(\frac{1}{2})}(\tau)$, resulting in the expansion
\[g^{(\frac{1}{2})}(\tau) = 2 C_2 \sum_{n \equiv (-1)^s \pmod{4}} a(n) e^{2\pi i n (\frac{2\tau+1}{8})}.\]

Examining this Fourier expansion, we see that the orders of vanishing of a modular form of half integral weight in the Kohnen plus space at the cusps $0, \frac{1}{2}$ of $\Gamma_0(4)$ are related to the Fourier expansion of the form at the cusp at $\infty$.  Specifically, if \[f(\tau) = \sum_{4n \geq n_0} a(4n) q^{4n} + \sum_{4n+(-1)^s \geq n_1} a(4n + (-1)^s) q^{4n+(-1)^s}\] with $a(n_0), a(n_1) \neq 0$, then
we have
\[\begin{array}{lcl}\textrm{Ord}_{\Gamma_0(4)}(f, \infty) &=& \min\{n_0, n_1\},\\
\\ \textrm{Ord}_{\Gamma_0(4)}(f, 0) &=& \frac{n_0}{16}\cdot 4 = \frac{n_0}{4},\\
\\ \textrm{Ord}_{\Gamma_0(4)}\left(f, \frac{1}{2}\right)& =& \frac{2n_1}{8} = \frac{n_1}{4}.\end{array}\]

For example, let $f_{k, m}(\tau)$ be a basis element for $M_k^!$, where $k = 12a + b + \frac{1}{2}$ and $b=6$.  The Fourier expansion is $f_{k, m}(\tau) = q^{-m} + O(q^{4+4a})$ and has exponents congruent to $0$ or $1 \pmod{4}$, and we note that $f_{k, m}(\tau)$ vanishes to order $-m$ at the cusp at $\infty$.  Thus, if $m$ is even, it follows that $\textrm{Ord}_{\Gamma_0(4)}(f, 0)$ is $-\frac{m}{4}$, while $\textrm{Ord}_{\Gamma_0(4)}(f, \frac{1}{2})$ is at least $\frac{4 + 4a + 1}{4} = a + \frac{5}{4}$, with equality if the coefficient of $q^{4+4a+1}$ is nonzero.  If this coefficient is zero, then the first odd exponent in the Fourier expansion is $4+4a+1+4\epsilon$ for some positive $\epsilon \in \Z$, which means that $\textrm{Ord}_{\Gamma_0(4)}(f, \frac{1}{2})$ is equal to $\frac{4+4a+4\epsilon + 1}{4} = a + \frac{5}{4} + \epsilon$.  If $m$ is odd, then $\textrm{Ord}_{\Gamma_0(4)}(f, \frac{1}{2}) = -\frac{m}{4}$, and $\textrm{Ord}_{\Gamma_0(4)}(f, 0)$ is equal to $1+a+\epsilon$ for some nonnegative integer $\epsilon$, with $\epsilon = 0$ only if the coefficient of $q^{4+4a}$ is nonzero.  Applying the valence formula, we find that there must be $5a + \frac{5m}{4} + \frac{b}{2} - 1 - \epsilon$ zeros in the fundamental domain if $m$ is even, and $5a + \frac{5m}{4} + \frac{b}{2} - \frac{3}{4} - \epsilon$ zeros in the fundamental domain if $m$ is odd.

Writing $k = 12a + b + \frac{1}{2}$ as before, similar computations for each value of $b$ show that $f_{k, m}(\tau)$ has $5a + \frac{5m}{4} + \frac{b}{2} - C - \epsilon$ zeros in $\mathcal{F}$, where the constant $C$ is as follows. \[\begin{array}{lcl} C&=&1 \textrm{ if } m \textrm{ is even and } b \textrm{ is even}; \medskip\\
C&=&\frac{3}{2} \textrm{ if } m \textrm{ is even and } b \textrm{ is odd};\medskip\\
C&=&\frac{3}{4} \textrm{ if } m \textrm{ is odd and } b \in \{6, 8, 9, 10, 11, 13\};\medskip\\
C&=&\frac{7}{4} \textrm{ if } m \textrm{ is odd and } b \in \{12, 14, 15, 16, 17, 19\}.\end{array}\]
Because $(-1)^{s-1} m \equiv 0, 1 \pmod{4}$, it can be checked that the number of zeros in $\mathcal{F}$ is always an integer, as expected.

Examining the Fourier expansion \[f_{k, m}(\tau) = q^{-m} + \mathop{\sum_{n > N}}_{(-1)^s n \equiv 0, 1\pmod{4}}a_k(m, n) q^n,\] we see that the quantity $\epsilon$ counts the number of initial consecutive coefficients which are equal to zero in (depending on the parity of $m$) either the sequence $\{a_k(m, N+4), a_k(m, N+8), a_k(m, N+12), \ldots\}$, or the appropriate sequence $\{a_k(m, N+1), a_k(m, N+5), a_k(m, N+9), \ldots\}$ or $\{a_k(m, N+3), a_k(m, N+7), a_k(m, N+11), \ldots\}$.

Since Theorem~\ref{mainthm} gives lower bounds on the number of zeros in the upper half plane, we can find upper bounds for $\epsilon$.  Specifically, we find that \[\epsilon \leq 3a + \frac{3m}{4} + \frac{b}{2} - C,\] so if $m$ is large enough compared to $\abs{a}$ it is impossible to have more than this many consecutive coefficients of even or odd powers of $q$ (starting from $q^N$) vanishing in the Fourier expansion of $f_{k, m}(\tau)$.  Because the Fourier coefficients $a_k(m, n)$ of the basis elements $f_{k, m}(\tau)$ satisfy the Zagier duality relation \[a_k(m, n) = -a_{2-k}(n, m),\] as proved in~\cite{DJ2}, we may apply this duality to give Theorem~\ref{lehmerthm}.

\section{Proof of Theorem~\ref{mainthm}} \label{sec3}
In this section, we prove Theorem~\ref{mainthm}, relying upon Theorem~\ref{thm_fintegral} and Theorem~\ref{prop_Akconstant}, which will be established in sections~\ref{sec_akm} and~\ref{pfsec2}. We begin with the generating function for the basis elements $f_{k,m}$, which is given by
\begin{align*} B_k(z; \tau) = \sum_m f_{k,m}(z) q^m  = \frac{f_k(z)f_{2-k}^*(\tau) + f_k^*(z)f_{2-k}(\tau)}{j(4\tau)-j(4z)},
\end{align*}
where $q=e^{2\pi i \tau}$.
The weakly holomorphic modular forms $f_k$ and $f_k^*$, defined in~\cite{DJ2}, are the first two elements of the canonical basis for $M_k^!$.

We fix $z$ to be on the lower left arc $\mathcal{A}$ of the fundamental domain $\mathcal F$, and write $z = -\frac{1}{4} +  \frac{e^{i\theta}}{4}$ for some $0 < \theta < \pi$.
By Cauchy's theorem, we have that
\begin{align}\label{eqn_cauchy} f_{k,m}(z) = \frac{1}{2\pi i}\int_{C_\epsilon} \frac{ B_k(z;\tau)}{q^{m+1}} dq =\int_{-\frac{1}{2} + i \epsilon}^{\frac{1}{2} + i \epsilon}   B_k(z;\tau) e(-m\tau) d\tau,\end{align}
where $C_\epsilon$ is a circular contour traversed counterclockwise about the origin, with radius $e^{-2\pi \epsilon}$ for some sufficiently large real number $\epsilon > \frac{1}{4}$.  Consider the contour
$C_v$ from $\tau = -\frac12 + i v$ to $\tau = \frac12 + iv$ for some fixed sufficiently small $v$, with $\epsilon>v>0$, and
let $W_{\epsilon,v}$ be the region bounded above by the horizontal line $-\frac{1}{2} + i\epsilon$ to $\frac{1}{2} + i\epsilon$, below by $C_v$, and on the sides by $-\frac12 + iv$ to $-\frac12+ i \epsilon$, and $\frac12 + i \epsilon$ to $\frac12+ iv$.
We have by the residue theorem and~(\ref{eqn_cauchy})  that
\begin{align}\label{eqn_resfkm}
f_{k,m}(z) = -2\pi i \sum_{w}  \displaystyle  \mathop{\textnormal{Res}}_{\tau = w} \left(B_k(z;\tau)e(-m\tau)\right)+ \int_{-\frac{1}{2} + i v}^{\frac{1}{2} + i v}   B_k(z;\tau) e(-m\tau) d\tau,\end{align}
where the sum above runs through the poles $w$ of the function $ B_k(z;\tau)e(-m\tau)$, when viewed as a function of $\tau$, in
the region $W_{\epsilon, v}$.
We explicitly determine these residue sums in Theorem~\ref{thm_fintegral} for all $z$ with $\theta \in [\frac{\pi}{3}, \frac{2\pi}{3}]$.  Using that result, we conclude that for $z$ along this arc, the (normalized) basis functions $f_{k,m}$ satisfy
\begin{align*} e^{\frac{ik\theta}{2}} e^{-\frac{\pi m}{2} \sin(\theta)} f_{k,m}(z) = 2 &\cos \left(\frac{k\theta}{2} + \frac{\pi m}{2} - \frac{\pi m \cos(\theta)}{2}\right) \\ & + \mathcal E_{m,k}(\theta) + e^{\frac{ik\theta}{2}} e^{-\frac{\pi m}{2} \sin(\theta)}\int_{-\frac12 + iv}^{\frac12 + iv} B_k(z;\tau) e(-m\tau) d\tau.\end{align*} Here the function $\mathcal E_{m,k}$ is equal to either $0, \mathcal C_{m,k}$ or $\mathcal D_{m,k}$, depending on the value of $\theta$; the functions $\mathcal C_{m,k}$ and $\mathcal D_{m,k}$ are defined in~(\ref{def_Cmk}) and~(\ref{def_Dmk}), respectively.

A short calculation, using the addition formula for cosine, reveals that
\begin{align}  \cos&\left(\frac{k\theta}{2} + \frac{\pi m}{2} - \frac{\pi m \cos(\theta)}{2}\right)  \label{cossincases}  = \begin{cases}-\sin\left(h_{k,m}(\theta)\right), & \text{ if } m\equiv 1 \!\!\!\pmod{4} \text{ and } k+\frac{1}{2} \equiv 0 \!\!\!\pmod{2}, \\
 \sin\left(h_{k,m}(\theta)\right), & \text{ if } m\equiv 3 \!\!\!\pmod{4} \text{ and } k+\frac{1}{2} \equiv 1 \!\!\!\pmod{2}, \\
 \cos\left(h_{k,m}(\theta)\right), & \text{ if } m\equiv 0 \!\!\!\pmod{4},\end{cases}
\end{align}
where \[h_{k,m}(\theta) := \frac{k\theta}{2} - \frac{\pi m \cos(\theta)}{2}.\]
The function $\sin(h_{k,m}(\theta))$ oscillates between $\pm 1$ when $h_{k,m}(\theta) = \frac{\pi}{2}(2n+1)$ and $n$ runs through $\mathbb Z$.  Similarly, $\cos(h_{k,m}(\theta))$ oscillates between $\pm 1$ when $h_{k,m}(\theta) = \pi n$ and $n$ runs through $\mathbb Z$.
Looking at the endpoints of the interval containing $\theta$, we have that
\begin{align*}h_{k,m}\left(\frac{ \pi}{3}\right) = \frac{\pi}{2}\left(\frac{k}{3} - \frac{m}{2}\right), \  \
h_{k,m}\left(\frac{2\pi}{3}\right) &=  \pi \left(\frac{k}{3} + \frac{m}{4}\right).
\end{align*}
Counting  odd integers in the interval \[\left[\frac{k}{3}-\frac{m}{2}, \frac{2k}{3} + \frac{m}{2}\right]\] and integers in the interval \[\left[ \frac{k}{6}-\frac{m}{4}, \frac{k}{3} + \frac{m}{4}\right],\] we conclude
that in all three cases of~(\ref{cossincases}) there are at least $2a + \frac{m}{2} + 1$ points on $\mathcal{A}$ with $\frac{\pi}{3} < \theta < \frac{2\pi}{3}$ where \[2 \cos\left(\frac{k\theta}{2} + \frac{\pi m}{2} - \frac{\pi m \cos(\theta)}{2}\right) = \pm 2.\]  More specifically, if $k = 12a + b + \frac{1}{2}$, we find that there are at least $2a + \lfloor \frac{m}{2} \rfloor + c$ points on the arc where the function is $\pm 2$, where $c$ depends on $b$ and the parity of $m$ and is given in Table~\ref{table1}.
\begin{table}
  \centering
  \caption{Values of $c$ for each $b$}\label{table1}
    \begin{tabular}{|r|c|c|c|c|c|c|c|c|c|c|c|c|}
    \hline
    $b$ & 6 & 8 & 9 & 10 & 11 & 12 & 13 & 14 & 15 & 16 & 17 & 19 \\
    \hline
    $c \, (m \textrm{ even})$ & 1 & 1 & 2 & 2 & 2 & 2 & 2 & 2 & 3 & 3 & 3 & 3 \\
    \hline
    $c \, (m \textrm{ odd})$ & 2 & 3 & 2 & 3 & 2 & 3 & 3 & 4 & 3 & 4 & 3 & 4 \\
    \hline
    \end{tabular}
\end{table}

By an argument similar to that in section~4 of~\cite{HJ}, we find that for any modular form $f \in M_k^!$, the quantity $e^{\frac{ik\theta}{2}} f(z)$ is real-valued for any $z \in \mathcal{A}$.  We note that the function $2\cos(\frac{\pi m}{2} + h_{k, m}(\theta))$ is real-valued as well, and prove in Lemma~\ref{cdboundlem} and section~\ref{secerrorbound} that the quantity \[\mathcal E_{m,k}(\theta) + e^{\frac{ik\theta}{2}} e^{-\frac{\pi m}{2} \sin(\theta)}\int_{-\frac12 + iv}^{\frac12 + iv} B_k(z;\tau) e(-m\tau) d\tau\] is bounded in absolute value by $2$.  Since the real-valued normalized modular form must be alternately positive and negative at $2a + \frac{m}{2} + 1$ points along $\mathcal{A}$, we apply the Intermediate Value Theorem to see that it must be zero at $2a + \frac{m}{2}$ distinct points on $\mathcal{A}$, and Theorem~\ref{mainthm} follows.

\section{Residue sums}\label{sec_akm}

Essential to our proof of Theorem~\ref{mainthm} in the previous section is an explicit determination of the residue sums in (\ref{eqn_resfkm}) for many different values of $z = -\frac{1}{4} + \frac{e^{i\theta}}{4}$ on the arc $\mathcal{A}$.   We do so in Theorem~\ref{thm_fintegral} below. To describe this result,
we define the functions
\begin{align}\label{def_Cmk} \mathcal C_{m,k}(\theta) &:= - c_{m,k}(2i)^{-k}\left(\sin(\tfrac{\theta}{2})\right)^{-k}e^{-\frac{\pi i m}{4}}
e^{\frac{\pi m}{2} \left(\frac{1}{2\tan\left(\frac{\theta}{2}\right)} - \sin(\theta)\right)}, \\ \label{def_Dmk}
 \mathcal D_{m,k}(\theta) &:= -  d_{m,k}\cdot 2^{-k} \left(\cos(\tfrac{\theta}{2})\right)^{-k}   e^{\frac{\pi i m}{4}} e^{\frac{\pi m}{2}\left( \frac{\tan\left(\frac{\theta}{2}\right)}{2} -\sin(\theta)\right) }, \end{align}
 where the constants $c_{m,k}$ and $d_{m,k}$ are defined by
 \begin{align}\label{def_cmkdef} c_{m,k} :=   \begin{cases}
1+i, & m\equiv 0 \pmod 4, \ k+\frac12 \equiv 1 \pmod 2, \\
1-i, & m\equiv 0 \pmod 4, \ k+\frac 12 \equiv 0 \pmod 2, \\
0, &  m\equiv 1 \pmod 4, \ k+\frac12 \equiv 0 \pmod 2, \\
0 & m\equiv 3 \pmod 4, \ k+\frac12 \equiv 1 \pmod 2.
\end{cases}
\end{align}
\begin{align}\label{def_dmkdef} d_{m,k} :=   \begin{cases}
1-i, & \text{ if } m\equiv 1 \ (\!\!\!\!\mod 4), k+\frac12 \equiv 0 \ (\!\!\!\!\mod 2), \\
1+i, & \text{ if } m\equiv 3 \ (\!\!\!\!\mod 4), k+\frac12 \equiv 1 \ (\!\!\!\!\mod 2), \\
0, & \text{ if } m\equiv 0 \ (\!\!\!\!\mod 4).
\end{cases}
\end{align}

Additionally, we define four smaller arcs $R_d$ ($1\leq d \leq 4$) contained in $\mathcal{A}$ by
\begin{align*}
 &R_1 :=\left\{   \frac{-1+e^{i\theta}}{4} \  \left| \  \theta \in \left[\frac{5\pi}{12},\frac{\pi}{2}\right) \right. \right\},  &   R_2  := \left\{  \frac{-1+e^{i\theta}}{4} \ \left|  \ \theta \in \left(\frac{\pi}{2},\frac{7\pi}{12}\right] \right. \right\}, \ \ \\
 &R_3 := \left\{ \frac{-1+e^{i\theta}}{4} \ \left|  \ \theta \in \left(\frac{\pi}{3},\frac{5\pi}{12}\right]\right.\right\},  &
R_4 := \left\{  \frac{-1+e^{i\theta}}{4} \ \left|  \ \theta\in \left[\frac{7\pi}{12},\frac{2\pi}{3}\right)\right.\right\}.  \end{align*}
In what follows, we choose $v = .2125$ when $z \in R_1 \cup R_2$ or $\theta = \frac{\pi}{2}$, and we choose $v = .1375$ when $z \in R_3 \cup R_4$.

\begin{thm} \label{thm_fintegral} With notation and hypotheses as above, we have that
\begin{align*} e^{\frac{ik\theta}{2}}&e^{-\frac{\pi m}{2}\sin(\theta)}\int_{-\frac{1}{2} + i v}^{\frac{1}{2} + iv} B_k(z;\tau) e(-m\tau) d\tau     \\ &=\begin{cases}    e^{\frac{ik\theta}{2}} e^{-\frac{\pi m}{2}\sin(\theta)}f_{k,m}(z)  - 2 \cos\left(\frac{k\theta}{2} + \frac{\pi m}{2} - \frac{\pi m \cos(\theta)}{2}\right), & \theta \in \left[\frac{5\pi}{12}, \frac{7\pi}{12}\right],  \\
  e^{\frac{ik\theta}{2}}e^{-\frac{\pi m}{2}\sin(\theta)}  f_{k,m}(z)  \!- 2\cos\left(\tfrac{k\theta}{2} + \tfrac{\pi m}{2} - \tfrac{\pi m \cos(\theta)}{2}\right) +  \mathcal C_{m,k}(\theta), & \theta \in \left(\frac{\pi}{3}, \frac{5\pi}{12}\right], \\
  e^{\frac{ik\theta}{2}}e^{-\frac{\pi m}{2}\sin(\theta)}  f_{k,m}(z)  \!- 2\cos\left(\tfrac{k\theta}{2} + \tfrac{\pi m}{2} - \tfrac{\pi m \cos(\theta)}{2}\right) +  \mathcal D_{m,k}(\theta), & \theta \in \left[\frac{7\pi}{12}, \frac{2\pi}{3}\right).
\end{cases}
\end{align*}
\end{thm}
Our proof of Theorem~\ref{thm_fintegral}, which begins with equation~(\ref{eqn_resfkm}), first requires that for a fixed value of $z$,
we determine the location of the poles of the function $B_k(z;\tau)$ lying in  the region $W_{\epsilon,v}$.
Since the denominator of $B_k(z; \tau)$ is $j(4\tau)-j(4z)$, its poles may occur only at points where this denominator is zero.  Thus, for each value of $z$ we examine the set \[ \left\{ \tau \in W_{\epsilon,v} \ \left| \  j(4\tau) = j(4z)\right.\right\}.\]  Fixing $z\in R_d$ and choosing the values of $v$ noted above,
we compute that the function $B_k(z;\tau)$ has either eight or twelve poles in this region.  They occur at points of the form $\tau = w_M(z)$, where $w_M(z) = \frac14 M(4z)$ for each matrix $M$ in a set $\mathcal{M}_d \subseteq \textnormal{SL}_2(\mathbb Z)$.
Here $M$ acts on points $z\in \mathcal{H}$ by $M(z) = \frac{az+b}{cz+d}$, as usual, and
the sets $\mathcal M_d $ are defined by
\begin{align*} \mathcal M_1&:=\{M_r^{(1)}, M_r^{(3)} \ | \  -1\leq r\leq 2\}, \\
\mathcal M_2&:=\{M_r^{(1)} \ | \ 0\leq r \leq 3\} \cup \{M_r^{(3)} \ | \ -2 \leq r\leq 1\}, \\
 \mathcal M_3&:=\{M_r^{(1)}, M_r^{(3)} \ | \ -1\leq r \leq 2\} \cup \{M_r^{(2)} \ | \ -2\leq r \leq 1\}, \\
 \mathcal M_4&:=\{M_r^{(1)} \ | \ 0\leq r \leq 3\} \cup \{M_r^{(3)} \ | \ -2\leq r \leq 1\} \cup \{M_r^{(4)} \ | \ -1\leq r \leq 2\}, \end{align*}
in terms of the integer matrices
\begin{align*} M_r^{(1)} :=   \sm{1}{r}{0}{1}, \
 M_r^{(2)}  :=  \sm{r}{-1}{1}{0},  \
 M_r^{(3)}  :=  \sm{r}{r-1}{1}{1}, \
  M_r^{(4)}  :=  \sm{r}{2r-1}{1}{2}.   \end{align*}
In the proof, we will need to compute specific values of the function \[A_k(z;\tau) :=  \frac{f_k(z)f_{2-k}^*(\tau) + f_k^*(z)f_{2-k}(\tau)}{\frac{d}{d\tau} j(4\tau)}\] when $\tau$ is equal to one of these poles $w_M(z)$ of $B_k(z;\tau)$.
These values are given in Proposition~\ref{prop_Akconstant} in section \ref{pfsec2}, in terms of the constants
\[\kappa_{r,k} :=
\begin{cases} \frac{-1}{8\pi i}(1+i^r), & r\equiv 0 \pmod{2}, \\
\frac{-1}{8\pi i}(1+i^r (-1)^{k+\frac12}), &   r\equiv 1 \pmod{2}.\end{cases}\]
We devote the remainder of this section to the proof of Theorem~\ref{thm_fintegral}.

\begin{proof}[Proof of Theorem~\ref{thm_fintegral}]
We divide the proof into five cases, depending on the value of $z$.  We begin by rewriting $B_k(z;\tau)$ as \begin{align*}  B_k(z;\tau) = A_k(z;\tau) \frac{d}{d\tau} \log \left(\widetilde{j}(\tau)\right), \end{align*}
and seek to determine \begin{align}\label{eqn_rescomp}  -2\pi i \mathop{\sum_{\tau = w_M(z) }}_{M\in\mathcal M_d}   \textnormal{Res}\left(e(-m\tau)A_k(z;\tau) \frac{d}{d\tau} \log \left(\widetilde{j}(\tau)\right) \right)  \end{align} for each $d\in\{1,2,3,4\}$.
The reason for rewriting $B_k(z;\tau)$ in this way is that $\widetilde{j}(\tau) := j(4\tau)-j(4z)$ has a simple zero at any of the  poles $\tau= w_M(z)$ described above,
which implies that its logarithmic derivative has simple poles at these points with residue equal to 1.  Combining this with Proposition~\ref{prop_Akconstant} allows us to then calculate (\ref{eqn_rescomp}) explicitly.
The single case $z=\frac{-1+i}{4}$, with $\theta=\frac{\pi}{2}$, is not of this nature, and must be treated slightly differently.   We address this in Case 5 below.
\medskip \ \\ {\emph{Case 1: $z\in R_1$.}}  For $z\in R_1$, the function $B_k(z;\tau)$ has eight poles in $W_{\epsilon,v}$  as described above.  From (\ref{eqn_rescomp}) and the discussion following, as well as Proposition~\ref{prop_Akconstant}, we find that
\begin{align}\nonumber & -2\pi i \mathop{\sum_{\tau = w_M(z) }}_{M\in\mathcal M_1}   \textnormal{Res}\left(e(-m\tau)A_k(z;\tau) \frac{d}{d\tau} \log \left(\widetilde{j}(\tau)\right) \right)  \\   &= -2\pi i \left(e(-m z) \sum_{r=-1}^2 e\left(-\frac{m r}{4}\right) \kappa_{r,k} + e(-m u)(4z+1)^{-k}\left(e\left(\frac{m}{4}\right) \kappa_{-1,k} + \kappa_{0,k} + e\left(-\frac{m}{4}\right) \kappa_{1,k}\right) \right) \\
\label{rescasei2} &= -2\pi i ( e(-mz) + e(-mu)(4z+1)^{-k})\left(e\left( \frac{m}{4}\right)\kappa_{-1,k} + \kappa_{0,k} + e\left(-\frac{m}{4}\right) \kappa_{1,k}\right).
\end{align}

Consider the set of ordered pairs \begin{align*} T  &:=
\left \{(m,k)\in \mathbb Z \times \left(\left.\tfrac{1}{2}+\mathbb Z\right) \  \right| \ \begin{array}{l}  m \equiv 1\pmod{4} \text{ and } k + \frac{1}{2} \equiv 0 \pmod{2} \\
 m \equiv 3\pmod{4} \text{ and } k + \frac{1}{2} \equiv 1 \pmod{2} \\
 m \equiv 0\pmod{4}
\end{array} \right\}.  \end{align*}  From the definition of $f_{k, m}$, we know that
the pair $(m,k)$ must be in the set $T$.  Using this fact,
a short calculation using the definition of $\kappa_{r,k}$ reveals that for $(m,k)\in T$, we have
\begin{align*}   e\left( \frac{m}{4}\right)\kappa_{-1,k} + \kappa_{0,k} + e\left(-\frac{m}{4} \right)\kappa_{1,k}   =
 \frac{-1}{2\pi i}.\end{align*}
Next we simplify, using that $z=-\frac{1}{4} + \frac{1}{4}e^{i\theta}$ and $u=\frac{z}{4z+1}$, to obtain
\begin{align*}
e(-mz) + e(-mu)(4z+1)^{-k} &= e^{-\frac{\pi i m}{2}(-1+\cos(\theta) + i\sin(\theta))} + e^{-\frac{\pi i m}{2}(1-\cos(\theta) + i\sin(\theta))}e^{-i k \theta} \\ &=2e^{\frac{\pi m \sin(\theta)}{2}}e^{-\frac{ik \theta}{2}} \cos\left(\frac{k\theta}{2} + \frac{\pi m}{2} - \frac{\pi m \cos(\theta)}{2}\right).
\end{align*}
Combining the above, we find that~(\ref{rescasei2}) reduces to
\begin{align}\label{eqn_casei3}
2e^{\frac{\pi m \sin(\theta)}{2}}e^{-\frac{ik \theta}{2}} \cos\left(\frac{k\theta}{2} + \frac{\pi m}{2} - \frac{\pi m \cos(\theta)}{2}\right).\end{align}
We combine this with (\ref{eqn_resfkm}) and multiply  through by $e^{\frac{-\pi m \sin(\theta)}{2}} e^{\frac{i k \theta}{2}}$, which yields
Theorem~\ref{thm_fintegral} in the case $\theta \in \left[\frac{5\pi}{12},\frac{\pi}{2}\right)$.
\medskip \ \\ {\emph{Case 2: $z\in R_2$.}}   We proceed as in Case 1, and begin with (\ref{eqn_rescomp}).   Here, when $z\in R_2$, the function $B_k(z;\tau)$ has eight poles in $W_{\epsilon,v}$  as described above. Note that the sets $\mathcal M_1$ and $\mathcal M_2$ which determine these poles in Case 1 and Case 2 (respectively) are  identical, save for  the matrix $\sm{1}{-1}{0}{1}$ which is replaced by $\sm{1}{3}{0}{1}$, and the matrix $\sm{2}{1}{1}{1}$ which is replaced by $\sm{-2}{-3}{1}{1}$.  It is not difficult to see that $e\left(\frac{m}{4}\right) = e\left(\frac{-3m}{4}\right)$, and that $\kappa_{-1,k} = \kappa_{3,k}$, which reveals that the residue contributions arising from $\sm{1}{-1}{0}{1}$ and $\sm{1}{3}{0}{1}$ are the same.  Moreover, for $w=w_{M_r^{(3)}}(z)$ where $r=\pm 2$, we have from Proposition~\ref{prop_Akconstant}
 that
$A_k(z,w)(4z+1)^k = \kappa_{1,k}$.  Thus, the residue contributions arising from $\sm{2}{1}{1}{1}$ and $\sm{-2}{-3}{1}{1}$ are also the same.  Thus, Case 2 is   identical to Case 1.  This yields
Theorem~\ref{thm_fintegral} in the case $\theta \in \left(\frac{\pi}{2},\frac{7\pi}{12}\right]$.
\medskip \ \\ {\emph{Case 3: $z\in R_3$.}}  With $z\in R_3,$  the function $B_k(z;\tau)$ has twelve poles in $W_{\epsilon,v}$ as described above.  Note that $\mathcal M_1\subset \mathcal M_3$,  thus, for those matrices in $\mathcal M_3 \cap \mathcal M_1 = \mathcal M_1$, we may use the calculation above in Case 1.  The remaining matrices in $\mathcal M_3\setminus \mathcal M_1$ give the following contribution to the total residue (\ref{eqn_rescomp}):
\begin{align}
-2\pi i &\mathop{\sum_{\tau = w_M(z) }}_{M\in\mathcal M_3 \setminus \mathcal M_1}e(-m \tau) A_k(z;\tau)\nonumber \\
&= -2 \pi i   (4z)^{-k}  e\left(\frac{m}{16z}\right)  \left[ \kappa_{-1,k}  +
e\left(  \frac{-m}{4}\right) \kappa_{-1,k}   + e\left(  \frac{m}{4}\right) i^{2k} \kappa_{1,k} +
e\left(  \frac{m}{2}\right)  i^{2k}\kappa_{1,k}\right]\nonumber \\
&= (4z)^{-k} \cdot e\left(\frac{m}{16z}\right)\cdot c_{m,k},\label{lastaddres}\end{align} where the constant $c_{m,k}$ is defined in (\ref{def_cmkdef}).

Using the fact that $z=-\frac{1}{4} + \frac{1}{4}e^{i\theta}$,  equation~(\ref{lastaddres}) becomes \begin{align}\nonumber
& c_{m,k} \cdot (-1 + e^{i\theta})^{-k} \cdot e\left(\frac{m}{4(-1 + e^{i\theta})}\right)
=  c_{m,k} \cdot (-1 + e^{i\theta})^{-k}e\left(\frac{m e^{\frac{-i\theta}{2}}}{8i \sin\left(\tfrac{\theta}{2}\right) }\right) \\ \label{bad} &= c_{m,k}\cdot (-1 + e^{i\theta})^{-k}e^{\frac{-\pi i m}{4}} e^{\frac{\pi m}{4\tan \left(\frac{\theta}{2}\right ) }}  = c_{m,k}\cdot (2 i)^{-k} \left(\sin\left(\tfrac{\theta}{2}\right)\right)^{-k} e^{\frac{-ik\theta}{2}} e^{\frac{-\pi i m}{4}} e^{\frac{ \pi m }{ 4\tan\left(\frac{\theta}{2}\right) }}.
\end{align}
We already showed that (\ref{rescasei2}) reduces to (\ref{eqn_casei3}).
Adding (\ref{bad}) and (\ref{eqn_casei3}) and then multiplying their sum by $e^{\frac{ik\theta}{2}} e^{\frac{-\pi m \sin(\theta)}{2}}$ proves Theorem~\ref{thm_fintegral} in the case $\theta \in \left(\frac{\pi}{3},\frac{5\pi}{12}\right]$.

\medskip \ \\ {\emph{Case 4: $z\in R_4$.}}  With $z\in R_4,$  the function $B_k(z;\tau)$ has twelve poles in $W_{\epsilon,v}$.
Note that $\mathcal M_2\subset \mathcal M_4$,  so for those matrices in $\mathcal M_4 \cap \mathcal M_2 = \mathcal M_2$, we may use the calculation above in Case 2, which we showed was identical to Case 1.  After some short calculations, proceeding as above, we find that the remaining matrices in $\mathcal M_4\setminus \mathcal M_2$ give the following contribution to the total residue (\ref{eqn_rescomp}):
\begin{align}
-2&\pi i  \mathop{\sum_{\tau = w_M(z) }}_{M\in\mathcal M_4 \setminus \mathcal M_2}e(-m \tau) A_k(z;\tau)\nonumber \\
&= -2 \pi i   (4z+2)^{-k}   e\left(\frac{m}{16z+8}\right) \!\! \left[ \kappa_{-1,k}  +
e\left(  \frac{m}{4}\right)i^{2k} \kappa_{-1,k}   + e\left(  -\frac{m}{4}\right)  \kappa_{1,k} +
e\left(  -\frac{m}{2}\right)e^{-\pi i k}  \kappa_{1,k}\right]\nonumber \\
&= (4z+2)^{-k} \cdot e\left(\frac{m}{16z+8}\right)\cdot d_{m,k},\label{lastaddres2}\end{align} where the constant $d_{m,k}$ is defined in (\ref{def_dmkdef}).
We rewrite $z=-\frac{1}{4} + \frac{1}{4}e^{i\theta}$ so that (\ref{lastaddres2}) becomes \begin{align}\nonumber
& d_{m,k} \cdot (1 + e^{i\theta})^{-k}e\left(\frac{m e^{\frac{-i\theta}{2}}}{4(e^{\frac{-i\theta}{2}} +
e^{\frac{i\theta}{2}})}\right) \\
\label{bad2} &= d_{m,k} \cdot (1 + e^{i\theta})^{-k}e\left(\frac{m e^{\frac{-i\theta}{2}}}{8\cos\left(\tfrac{\theta}{2}\right) }\right) = d_{m,k}\cdot 2^{-k} \left(\cos\left(\tfrac{\theta}{2}\right)\right)^{-k} e^{\frac{-ik\theta}{2}} e^{\frac{\pi i m}{4}} e^{\frac{\pi m \tan\left(\frac{\theta}{2}\right)}{4}}.
\end{align}
Adding (\ref{bad2}) and (\ref{eqn_casei3}) and multiplying their sum by $e^{\frac{ik\theta}{2}} e^{\frac{-\pi m \sin(\theta)}{2}}$ yields
Theorem~\ref{thm_fintegral} in the case $\theta \in \left[\frac{7\pi}{12},\frac{2\pi}{3}\right).$
\medskip \ \\ {\emph{Case 5: $z=\frac{-1+i}{4} \ (\text{i.e. } \theta=\frac{\pi}{2}).$}}
We define for integers $n\geq 4$   \begin{align*} \theta_n^{\pm}  := \frac{\pi}{2} \pm \frac{1}{n}, \  \
z_n^{\pm}  := -\frac{1}{4} + \frac{e^{i \theta_n^{\pm}}}{4}.\end{align*}
  Then each $z_n^+ \in R_2$ and $z_n^- \in R_1$.   From Case 1 and Case 2 of Theorem~\ref{thm_fintegral} established above, we have that \begin{align}\label{eqn_pm} e^{\frac{ik\theta_n^\pm}{2}} e^{-\frac{\pi m}{2}\sin(\theta_n^\pm)} f_{k,m}(z)    - 2 \cos&\left(\frac{k\theta_n^\pm}{2} + \frac{\pi m}{2} - \frac{\pi m \cos(\theta_n^\pm)}{2}\right)    = \int_{\mathcal I} h_n^\pm(\tau)d\tau, \end{align}
where the interval $\mathcal I$ is equal to $[-\frac12 + i v ,\frac12 + i v ]$ and the functions $h_n^\pm(\tau)$ and $F_k$ are defined by \begin{align*}
h_n^\pm(\tau) &= h_{n,m,k}^{\pm}(\tau) := e^{\frac{ik\theta_n^\pm}{2}}e^{-\frac{\pi m}{2}\sin(\theta_n^\pm)} \frac{F_k(z_n^\pm;\tau)e(-m\tau)}{j(4\tau) - j(4z_n^\pm)},
\end{align*}
\begin{align}\label{def_Fkzt}F_k(z;\tau) := f_k(z)f_{2-k}^*(\tau) + f_k^*(z)f_{2-k}(\tau).\end{align} We have that $ \lim_{n\to \infty} h_n^{\pm}(\tau) = h(\tau)$ on $\mathcal I$, where \[h(\tau) = h_{m,k}(\tau)  := e^{\frac{ik\pi}{4}}e^{-\frac{\pi m}{2} } \frac{F_k(-\tfrac14 + \tfrac{i}{4};\tau)e(-m\tau)}{j(4\tau) - 1728} .\]  By construction, $h_n^{\pm}(\tau)$ and $h(\tau)$ are integrable on $\mathcal I $, one boundary of our rectangular region $W_{\epsilon,v}$.

Consider the function $H(\theta,t)$ defined on the domain $\mathcal J := \left[\frac{\pi}{2}-\frac14,\frac{\pi}{2}+\frac14\right] \times \left[-\frac12,\frac12\right]$ by
\[ H(\theta,t) = H_{k,m,\epsilon}(\theta,t):= e^{\frac{ik\theta}{2}}e^{-\frac{\pi m}{2}\sin(\theta)} \frac{F_k\left(-\frac14 + \frac{e^{i\theta}}{4};t + i\epsilon \right)e(-m(t+i\epsilon))}{j(4(t+i\epsilon)) - j\left(4\left(-\frac14 + \frac{e^{i\theta}}{4}\right)\right)}.\]  Note that for integers $n\geq 4$, we have that $\theta_n^{\pm} \in \left[\frac{\pi}{2}-\frac14,\frac{\pi}{2}+\frac14\right]$, and that $H(\theta_n^{\pm},t) = h_n^{\pm}(\tau)$ when $\tau = t+iv$  is on the path of integration $\mathcal I$ (i.e. $t \in [-\frac12,\frac12]$).  By construction, the function $H(\theta,t)$ has no poles on $\mathcal J$ and is continuous, so is bounded by some constant $K$ on this region.  Therefore, as a function of $t$, for fixed $\theta_n^{\pm}$, the function $H(\theta_n^{\pm},t)$ is
bounded by $K$ for $t\in [-\frac12,\frac12]$, and hence $|h_n^\pm(\tau)|\leq K$ for all $\tau \in \mathcal I  $.  By the Bounded Convergence Theorem and (\ref{eqn_pm}), we have that
\[ e^{\frac{ik\pi}{4}} e^{-\frac{\pi m}{2} } f_{k,m}\left(-\tfrac14 + \tfrac{i}{4}\right)    - 2 \cos \left(\tfrac{k\pi}{4} + \tfrac{\pi m}{2}  \right)  =   \lim_{n\to \infty} \int_{-\frac12 + i v }^{\frac12+iv } h_n^{\pm}(\tau) d\tau =  \int_{-\frac12 + i v }^{\frac12+iv} h(\tau) d\tau. \]     Thus, Theorem~\ref{thm_fintegral} holds for $\theta=\frac{\pi}{2}$.
\end{proof}

\subsection{The integral weight case}
It is interesting to compare this approximation for $f_{k, m}(z)$ at $\theta = \frac{\pi}{2}$ with a similar computation for modular forms of integral weight in a canonical basis.  For an even integer $k$, write $k = 12\ell + k'$, where $k' \in \{0, 4, 6, 8, 10, 14\}$.  It was shown in~\cite{DJ1} that there exists a basis $\{f_{k, m}\}_{m \geq -\ell}$ for the space $M_k^!$, where $f_{k, m}(\tau) = q^{-m} + O(q^{\ell+1})$, and that the zeros of these basis elements lie on the unit circle when $m$ is large enough.  The proof depends on the contour integral
\[f_{k, m}(z) = \int_{-\frac{1}{2}+iA}^{\frac{1}{2}+iA} \frac{\Delta(z)^\ell}{\Delta(\tau)^\ell} \frac{E_{k'}(z)}{E_{k'}(\tau)} \frac{E_{14}(\tau)}{\Delta(\tau)} \frac{1}{j(\tau)-j(z)} q^{-m} d\tau = \int_{-\frac{1}{2}+iA}^{\frac{1}{2}+iA} G(\tau, z) d\tau.\]
Fixing $z = e^{i\theta}$ for $\theta \in \left(\frac{\pi}{2}, \frac{2\pi}{3}\right)$, the contour can be moved downward from $A > 1$ to pick up residues at $\tau = z$ and $\tau = -\frac{1}{z}$; these residues can be computed using the fact that $\frac{d}{d\tau} j(\tau) = \frac{E_{14}(\tau)}{\Delta(\tau)}$ to give
\[\int G(\tau, z) d\tau = f_{k, m}(z) - e^{-2\pi i m z} - z^{-k}e^{\frac{2\pi i m}{z}}.\]
Multiplying by $e^{\frac{ik\theta}{2}}e^{-2\pi m \sin(\theta)}$, we find that $e^{\frac{ik\theta}{2}}e^{-2\pi m \sin(\theta)} f_{k, m}(e^{i\theta}) - 2\cos\left(\frac{k\theta}{2} - 2\pi m \cos(\theta)\right)$ is equal to the weighted integral.

It is natural to ask what happens when $\theta = \frac{\pi}{2}$, when the denominator of the integrand has only one zero on the unit circle, at $i$, instead of the two zeros at $e^{i\theta}, e^{i(\pi-\theta)}$ from before.  Ideally, the function approximating $f_{k, m}$ should be continuous, so the cosine term should become
\[-2 \cos\left(\frac{k\pi}{4}\right) = \begin{cases} 0 & \mbox{if } k \equiv 2 \pmod{4}, \\ -2 & \mbox{if } k \equiv 0 \pmod{8}, \\ 2 & \mbox{if } k \equiv 4 \pmod{8}. \end{cases}\]

In this integral weight case,  continuity can be proved without invoking the Bounded Convergence Theorem.  In the $k \equiv 2 \pmod{4}$ case, $E_{k'}(z)$ has a zero at $z = e^{\frac{i\pi}{2}} = i$, and fixing $z = i$ means that the entire integrand is equal to zero; therefore, the residue is zero.  For the $k \equiv 0 \pmod{4}$ cases, we note that at the point $\tau = i$, the function $\frac{d}{d\tau} j(\tau) = \frac{E_{14}(\tau)}{\Delta(\tau)}$ has a simple zero, implying that $j(\tau) - j(i)$ has a double zero.  Thus, the logarithmic derivative
\[ \frac{E_{14}(\tau)}{\Delta(\tau)} \frac{1}{j(\tau)-j(z)} = \frac{\frac{d}{d\tau} (j(\tau)-j(z))}{j(\tau)-j(z)}\] contributes a factor of $2$ to the residue due to the double zero, implying that we get a residue
of either
$-2e^{-2\pi i m z}$ or $-2 z^{-k} e^{\frac{2\pi i m}{z}}$ at this point.  These expressions are equal if $k \equiv 0 \pmod{4}$, and
after multiplying by $e^{\frac{i k \pi}{4}} e^{-2\pi m \sin(\frac{\pi}{2})}$, we get the value of
$\pm2$ that we expected.
Thus, the approximation of $f_{k, m}(z)$ is continuous at $\theta = \frac{\pi}{2}$.

This continuity argument does not seem to work when $k \in \Z + \frac{1}{2}$. In the integral weight case, fixing a value of $z$ results in a quotient of modular forms in the variable $\tau$ inside the integral, with a form of weight $2$ in the numerator and a form of weight $0$ in the denominator, and this ratio can be written as a logarithmic derivative, allowing the computation of residues.  For half integral weight, the ratio of modular forms in $\tau$ is of weight $2-k$ in the numerator and weight $0$ in the denominator, and does not simplify to a logarithmic derivative in the same way.  At the poles of $B_k(z; \tau)$ in the region $W_{\epsilon, v}$, though, the ratio $A_k(z; \tau)$ of the numerator to the derivative of $j(4\tau)$ is constant, and we can compute its values and use the logarithmic derivative to compute residues as before.  However, when $\theta = \frac{\pi}{2}$, the derivative of $j(4\tau)$ is $0$, and the argument breaks down.

\section{Evaluating the function \texorpdfstring{$A_k(z;\tau)$}{Ak}}\label{pfsec2}
 As we have seen in the previous section, our proof of Theorem~\ref{thm_fintegral} requires specific values of the function \[A_k(z;\tau) :=  \frac{f_k(z)f_{2-k}^*(\tau) + f_k^*(z)f_{2-k}(\tau)}{\frac{d}{d\tau} j(4\tau)}\] when $\tau$ is specialized to be among these poles $w_M(z)$ of $B_k(z;\tau)$.
We state these required results in Proposition~\ref{prop_Akconstant} below, in terms of the constants
\[\kappa_{r,k} :=
\begin{cases} \frac{-1}{8\pi i}(1+i^r), & r\equiv 0 \pmod{2}, \\
\frac{-1}{8\pi i}(1+i^r (-1)^{k+\frac12}), &   r\equiv 1 \pmod{2}.\end{cases}\]
For a function $f:\mathbb H\to \mathbb C$, we let $Z(f) := \{z \in \mathbb H \ | \ f(z) = 0\}$.
 \begin{proposition}\label{prop_Akconstant} Let $k \in \frac{1}{2} + \mathbb Z,$ $r\in \mathbb Z$, and let $z \in \mathbb H \setminus Z\left(\frac{d}{d\tau} j(4\tau)\right)$.  We have that
\begin{align*} A_k(z;w_{M_r^{(1)}}(z) ) &=  \kappa_{r,k}, \\
A_k(z;w_{M_r^{(2)}}(z) ) &= (4z)^{-k}\cdot \begin{cases}\kappa_{-1,k}, &  r=0,1, \\
 i^{2k} \kappa_{1,k}, & r=-1,-2,
\end{cases} \\
A_k(z;w_{M_r^{(3)}}(z) ) &=  (4z+1)^{-k} \cdot \begin{cases} \kappa_{0,k}, & r=1, \\
 0, & r=-1, \\
\kappa_{1,k}, & r=\pm 2, \\
\kappa_{-1,k}, & r =0, \end{cases}  \\
A_k(z;w_{M_r^{(4)}}(z) ) &= (4z+2)^{-k} \cdot \begin{cases} \kappa_{1,k}, & \text{if } r=1, \\
\kappa_{-1,k}, & \text{if } r=0,\\
i^{2k}\kappa_{-1,k}, & \text{if } r=-1, \\
e^{-\pi i k} \kappa_{1,k}, & \text{if } r=2.
 \end{cases}
\end{align*}
\end{proposition}

In section~\ref{sec_akpart1}, we first establish the Proposition in the case $\tau = w_{M_r^{(1)}}(z)$.  We then use that result to establish the Proposition in the remaining cases in section~\ref{sec_akpart2}.
\subsection{Proof of Proposition~\ref{prop_Akconstant} part 1: $\tau = w_{M_r^{(1)}}(z)$} \label{sec_akpart1}
 To prove Proposition~\ref{prop_Akconstant} in the case $\tau = w_{M_r^{(1)}}(z)$ for any $r\in\mathbb Z$, our starting point begins with work of Duke and the second author, who show in \cite{DJ2} for any
 any $k = s + \frac12, s \in \mathbb Z$  that
 \[f_k(z) = \Delta(4z)^a f_{b+\frac{1}{2}}(z), \ \ \ f_k^*(z)=\Delta(4z)^a f^*_{b+\frac{1}{2}}(z),\] where $s=b+12a,$ for some $a\in \mathbb Z$  and $b \in S$, where
 \[S:= \{6,8,9,10,11,12,13,14,15,16,17,19\}.\]  The 24 forms $f_{b+\frac12}(z), f_{b+\frac12}^*(z)$ for $b\in S$ are given in the Appendix of \cite{DJ2}, as explicit polynomials in the weight 2 Eisenstein series $F$ on $\Gamma_0(4)$, and the weight $\frac12$ modular theta function $\vartheta$.
 Thus, writing $k=12 a + b + \frac12$ in this way, a short calculation reveals that we may write the function $F_k$ from (\ref{def_Fkzt}) as
\begin{align}\label{fkztrewrite}F_k(z;\tau) = \Delta(4z)^{a}\Delta(4\tau)^{-2-a} g_b(z;\tau),\end{align}
 where
\[g_b(z;\tau) := f_{b+\frac12}(z)f^*_{25-b+\frac12}(\tau) + f_{b+\frac12}^*(z)f_{25-b+\frac12}(\tau).\]  We let $\tau = w_{M_r^{(1)}}(z) = z + \frac{r}{4},$ and find from (\ref{fkztrewrite}), and the fact that $\Delta(z)=\Delta(z+1)$, that
 \begin{align}\label{eqn_Fgbj}F_k\left(z;z+\tfrac{r}{4}\right) & =\Delta(4z)^{-2}g_{b,r}(z),
\end{align} where for $b\in S$ and $r\in\mathbb Z$ we define  $g_{b,r}(z) := g_b\left(z;z+\tfrac{r}{4}\right).$ To study the functions $g_{b,r}(z)$ for any $r\in\mathbb Z$, and hence the functions $F_k(z;z+\tfrac{r}{4})$ by (\ref{eqn_Fgbj}), we claim that it suffices to assume $r\in\{0,1,2,3\}$.   Indeed,
each $f_k(z)$ has a Fourier expansion of the form \begin{align}\label{fkstarFE} \sum_{(-1)^s n \equiv 0,1 \pmod{4}} c_k(n) \ e(nz).\end{align}
Each $f_k^*(z)$  also has an expansion of the same form.   If $r$ and $r'$ are integers satisfying $r\equiv r' \pmod 4$, then for any $n\in \mathbb Z$, we have that $e\left(n\left(z+\frac{r}{4}\right)\right) = e\left(n\left(z+ \frac{r'}{4}\right)\right)$, and hence, by the definition of $g_{b,r}(z)$ we have that $g_{b,r}(z) = g_{b,r'}(z)$.

In what follows,  we consider the cases $r\in\{0,1,2,3\}$ separately.  We elaborate on the cases $r=0$ and $r=2$;  the cases $r=1$ and $r=3$ follow similarly.
Recall that elements in $M_k^!$ transform with character $\psi_k = \psi_{k,\gamma} := \rho^{2k}$ under $\gamma = \sm{a}{b}{c}{d} \in \Gamma_0(4)$, where $\rho=\rho_\gamma$ is defined in section \ref{sec2}.
   \medskip \ \\ {\emph{Case $r=0$.}}  By definition, we have that $g_{b,0}(z) \in M_{26}\left(\Gamma_0(4),\psi_{b+\frac12}\psi_{25-b+\frac12}\right)$.  However, $\psi_{b+\frac12}\psi_{25-b+\frac12} = \psi_{b+\frac12,\gamma}\psi_{25-b+\frac12,\gamma} =\left(\left(\frac{c}{d}\right)\varepsilon_d^{-1}\right)^{52}$ is the trivial character,   hence, $g_{b,0}(z) \in M_{26}(\Gamma_0(4))$, and thus, $F_k(z,z) \in M_{2}(\Gamma_0(4))$.  We also have that the function $\frac{d}{dz} j(4z) \in M_2(\Gamma_0(4))$.    The Sturm bound \cite{Sturm} for $M_2(\Gamma_0(4))$ is $[\textnormal{SL}_2(\mathbb Z) : \Gamma_0(4)] \cdot \frac{2}{12} = 1$, and thus, after directly checking Fourier expansions for each of the 12 possible choices of $b\in S$, we may conclude that $(-4\pi i)F_k(z,z) = \frac{d}{dz}j(4z)$, and hence, Proposition~\ref{prop_Akconstant} holds in the case $r=0$, for any $k \in \frac12 + \mathbb Z$.  Since $\kappa_{r,k} = \kappa_{r',k}$ by definition (and $g_{b,r}(z)=g_{b,r'}(z)$ as argued above) whenever $r\equiv r' \equiv 0 \pmod{4}$, we conclude that Proposition~\ref{prop_Akconstant} holds for $\tau = w_{M_r^{(1)}}(z)$ for any $r\equiv 0 \pmod{4}$ and any $k\in \frac12+ \mathbb Z$.  \medskip \ \\
{\emph{Case $r=2$.}}  Each $f_k(z)$ has a Fourier expansion as in (\ref{fkstarFE}); thus, we have that
\begin{align*} f_k\left(z+\tfrac{1}{2}\right)  = \!\!\!\!\sum_{(-1)^s n \equiv 0 \pmod{4}}\!\!\!\! c_k(n) e(nz) \ \ \   - \!\!\sum_{(-1)^s n \equiv 1 \pmod{4}} \!\!\!\!c_k(n) e(nz)
 = f_k(z) - 2 f_{k,\chi_2}(z),
\end{align*}
where \[f_{k,\chi_2}(z) := \sum_{(-1)^s n \equiv 0,1 \pmod{4}} c_k(n) e(nz) \chi_2(n) = \sum_{(-1)^s n \equiv 1 \pmod{4}} c_k(n) e(nz),\] and $\chi_2(n):= \frac12(1-(-1)^n)$ is the Dirichlet character $\pmod{2}$.
Since $f_k(z) \in M_k(\Gamma_0(4),\psi_k)$, it is a result from the classical theory of modular forms \cite{Koblitz} that  $f_{k,\chi_2}(z) \in M_k(\Gamma_0(16),\psi_k \chi_2^2) = M_k(\Gamma_0(16),\psi_k  )$. Thus, we deduce that $f_k\left(z+\frac12\right) \in M_k(\Gamma_0(16),\psi_k)$.  The above argument also applies to the functions $f_k^*(z)$, so we have that $f_k^*\left(z+\frac12\right) \in M_k(\Gamma_0(16),\psi_k)$ as well.
Using these facts, together with the definition of $g_{b,2}(z)$, we deduce that $g_{b,2}(z)$ is a form in the space $M_{26}(\Gamma_0(16),\psi_{b+\frac12}\psi_{25-b+\frac12}) =  M_{26}(\Gamma_0(16)).$   The Sturm bound for the subgroup $\Gamma_0(16)$ is $[\textnormal{SL}_2(\mathbb Z) : \Gamma_0(16)] \cdot \frac{2}{12} = 4;$ hence, the truth of Proposition~\ref{prop_Akconstant} in the case $r=2$ for any $k \in \frac12 + \mathbb Z$ follows after computing and comparing the Fourier expansions of each of the 12 possible functions $g_{b,2}(z)$ with the function $\frac{d}{dz} \ j(4z) $ as in the previous case.   Since $\kappa_{r,k} = \kappa_{r',k}$ by definition (and $g_{b,r}(z)=g_{b,r'}(z)$ as argued above) whenever $r\equiv r' \equiv 2 \pmod{4}$, we conclude that Proposition~\ref{prop_Akconstant} holds for $\tau = w_{M_r^{(1)}}(z)$ for any $r\equiv 2 \pmod{4}$ and any $k\in \frac12+ \mathbb Z$.   \medskip  \ \\ {\emph{Case $r=1$.}}  As above, we may re-write
\begin{align*} f_k\left(z+\tfrac{1}{4}\right)
 = \!\!\!\!\!\!\!\!\sum_{(-1)^s n \equiv 0 \pmod{4}}\!\! \!\!\!\!c_k(n) e(nz)    + i(-1)^s \!\!\!\!\!\!\!\!\!\!\!\sum_{(-1)^s n \equiv 1 \pmod{4}} \!\!\!\!\!\!c_k(n) e(nz)
= f_k(z) +(i(-1)^s -1) f_{k,\chi_2}(z).  \end{align*} The truth of Proposition~\ref{prop_Akconstant} in this case follows as in the previous case.  We point out that the term $(-1)^{k+\frac12}$ appearing in the definition of $\kappa_{r,k}$ for odd $r$ depends only on $b$.  That is, $(-1)^{k+\frac12} = (-1)^{b+1}$, so in this case (and in the case $r=3$ below), the 12 possible functions $g_{b,1}$ yield two different constants $\kappa_{1,k}$, depending on whether $b$ is even or odd.
\medskip  \ \\ {\emph{Case $r=3$.}}  As above, we may re-write
\begin{align*} f_k\left(z+\tfrac{3}{4}\right)
 = \!\!\!\!\!\!\sum_{(-1)^s n \equiv 0 \pmod{4}} \!\!\!\!\!\!c_k(n) e(nz) -i(-1)^s \!\!\!\!\!\!\!\!\!\!\!\sum_{(-1)^s n \equiv 1 \pmod{4}} \!\!\!\!\!\!c_k(n) e(nz)
 = f_k(z) -(i(-1)^s +1) f_{k,\chi_2}(z).
\end{align*} The truth of Proposition~\ref{prop_Akconstant} in this case follows as in the previous two cases.

\subsection{Proof of Proposition~\ref{prop_Akconstant} part 2:  $\tau = w_{M_r^{(d)}}(z), d\in\{2,3,4\}$} \label{sec_akpart2}
Our proof of the Proposition in this case makes uses of the Proposition in the case $d=1$, which we established in the previous section.  We divide our proof into three cases corresponding to whether  $d$ is equal to $2,3$ or $4$.
\medskip  \
\\ {\emph{Case $d=2$.}} We let $M=M_r^{(2)}.$
It is not difficult to see that we may rewrite
\[ w = w_{M_r^{(2)}}(z) = \begin{cases} \sm{1}{0}{4}{1}\left(z-\frac14\right), &  r=1, \\
\sm{1}{0}{4}{1}\left(u-\frac14\right), &  r=0, \\
\sm{-1}{0}{4}{-1}\left(z+\frac14\right), & r=-1, \\
\sm{-1}{0}{4}{-1}\left(u+\frac14\right), &  r=-2,
\end{cases} \]
where here and throughout this section, \[u =u(z) := \frac{z}{4z+1} = \sm{1}{0}{4}{1}z.\]  Using the fact that the functions $f_k$ and $f_k^*$ are in the space
$M_{k}\left(\Gamma_0(4),\psi_k\right)$,
 we find after a short calculation (using also that $z=\sm{1}{0}{-4}{1} u$) that
\begin{align}\label{Fk1}F_k(z;w) = \begin{cases} (4z)^{2-k} F_k\left(z;z- \frac{1}{4}\right), & r= 1, \\
(4u)^{2}(4z)^{-k} F_k\left(u;u - \frac{1}{4}\right), & r=0, \\
(4z)^{2-k} i^{2k} F_k\left(z;z+\frac{1}{4}\right), & r=-1, \\
(4u)^2(4z)^{-k} i^{2k} F_k\left(u;u +  \frac{1}{4}\right), &r=-2.
 \end{cases}
\end{align}   For ease of notation, we let  \[y=y_r(z) := \begin{cases} z, & r\in\{-1,1\}, \\ u, &r\in\{0,-2\}. \end{cases} \]
Using the modular properties of the function $j(\tau)$, and making a change of variable  ($d\tau = dw = (4y)^{-2} dy$), we see that
\begin{align}\label{j1}\frac{d}{d\tau} j(4\tau) = \frac{d}{d w} j(4 w) = (4y)^2 \frac{d}{dy} j(4y) = (4y)^2 \frac{d}{d(y\pm \frac{1}{4})} j(4(y\pm \tfrac{1}{4})),\end{align} where the sign above is taken to be $+$ if $r=-1,-2,$ and $-$ if $r=0,1$.
Combining (\ref{Fk1}) and (\ref{j1}), we find that
\[A_k(z;w) =  \begin{cases} (4z)^{-k} A_k\left(z;z- \frac{1}{4}\right), & r= 1, \\
(4z)^{-k}  A_k\left(u;u - \frac{1}{4}\right), & r=0, \\
(4z)^{-k} i^{2k} A_k\left(z;z+\frac{1}{4}\right), & r=-1, \\
 (4z)^{-k} i^{2k} A_k\left(u;u +  \frac{1}{4}\right), &r=-2.
 \end{cases}
\]
 By Proposition~\ref{prop_Akconstant} in the case $\tau = z + \frac{r}{4}, r\in\mathbb Z,$ proved in the previous section, we  have that $A_k\left(y;y\pm \frac14\right)$ is constant.  In particular, we conclude in this case that \begin{align}\label{eqn_Akt2}A_k(z;w_{M_r^{(2)}}(z) ) = (4z)^{-k}\cdot \begin{cases}\kappa_{-1,k}, &  r\in\{0,1\}, \\
 i^{2k} \kappa_{1,k}, & r\in\{-1,-2\}.
\end{cases}\end{align}
\medskip
\\ {\emph{Case $d=3$.}} We let $M=M_r^{(3)}.$
Proceeding as in the previous case, we rewrite
\[w=w_{M_r^{(3)}}(z) =\begin{cases} \sm{1}{0}{4}{1} z, & r=1, \\ \sm{-1}{-1}{4}{3} \left(z-\frac12\right), &  r=-1,\\
u \pm \frac{1}{4}, & r\in\{0,2\},\\
u - \frac{3}{4}, & r=-2.  \end{cases} \]
Again, using modular properties of $F_k(z;\tau)$,  we find that
\begin{align} \label{Fk2} F_k(z;w) = \begin{cases} (4z+1)^{2-k} F_k(z;z), & r=1, \\
(4z+1)^{2-k} i^{2k}F_k\left(z;z-\frac12\right), & r=-1,\\
(4z+1)^{-k} F_k\left(u;u\pm \frac14\right), & r\in\{0,2\},\\
(4z+1)^{-k} F_k\left(u;u-\frac34\right), & r=-2.\end{cases} \end{align}
We also compute in these cases that
\begin{align}\label{j2}
\frac{d}{dw}j(4w) = \begin{cases} (4z+1)^2 \frac{d}{dz} j(4z), & r = \pm 1, \\
\frac{d}{du} j(4 u), & r= 0,\pm 2.
\end{cases}  \end{align}Combining (\ref{Fk2}) and (\ref{j2}),  we have that
\[A_k(z;w) = (4z+1)^{-k} \begin{cases} A_k(z;z), & r=1, \\ i^{2k}A_k\left(z;z-\frac12\right), & r=-1, \\
A_k\left(u;u\pm \frac14\right), & r=2,0 \\
A_k\left(u;u-\frac34\right), & r=-2.\end{cases}
\]  Once again we combine Proposition~\ref{prop_Akconstant} in the case $\tau=z+\frac{r}{4}, r\in\mathbb Z,$ established in the previous section  with the above, and see that in all cases $A_k(z;\tau)$ is explicitly a  constant multiple of a rational function in $z$.   Precisely, we have that
\begin{align}\label{eqn_Akt3}A_k(z;w_{M_r^{(3)}}(z) ) = (4z+1)^{-k} \cdot \begin{cases} \kappa_{0,k}, & r=1, \\
i^{2k}\kappa_{-2,k} = 0, & r=-1, \\
\kappa_{1,k}, & r\in\{- 2,2\}, \\
\kappa_{-1,k}, & r =0. \end{cases}.\end{align}
\medskip
\\ {\emph{Case $d=4$.}} We let $M=M_r^{(4)}.$   Proceeding as above, we rewrite
\[w=\begin{cases} \sm{1}{0}{4}{1}\left(z+\frac14\right), & r=1, \\ \sm{-1}{0}{4}{-1}\left(z+\frac34\right), & r=-1, \\ \sm{1}{0}{-4}{1}\left(u-\frac14\right), & r=0, \\ \sm{1}{0}{4}{1}\left(u-\frac34\right), & r=2.
\end{cases} \]  Using modular properties of $F_k(z;\tau)$,  we find that \begin{align}\label{Fk3} F_k(z;w) = \begin{cases} (4z+2)^{2-k} F_k\left(z;z+\frac14\right), & r = 1,\\
(-4u+1)^{k}(4u-2)^{2-k} F_k\left(u;u-\frac34\right), &r=2, \\
(4z+2)^{2-k}i^{2k} F_k\left(z;z+\frac34\right), & r=-1, \\
(-4u+1)^{k}(-4u+2)^{2-k} F_k\left(u;u-\frac14\right), & r=0.
\end{cases}
\end{align}
We also compute  that
\begin{align}\label{j3} \frac{d}{dw}j(4w) = \begin{cases}
(4z+2)^2 \frac{d}{dz} j(4z), & r\in\{-1,1\},  \\
(4u-2)^2 \frac{d}{dz} j(4u), & r\in\{ 0,2\}.
\end{cases}
\end{align}
Combining (\ref{Fk3}) and (\ref{j3})  we find that
\[A_k(z;w) = (4z+2)^{-k} \begin{cases} A_k\left(z;z+\frac14\right), & r=1, \\
i^{2k}A_k\left(z;z+\frac34\right), & r=-1, \\
 A_k\left(u;u-\frac14\right), & r=0, \\
e^{-\pi i k} A_k\left(u;u-\frac34\right), & r=2.
\end{cases}
\]  Once again using Proposition~\ref{prop_Akconstant} in the case $\tau=z+\frac{r}{4}, r\in\mathbb Z,$ we see that in all cases $A_k(z;\tau)$ is explicitly a  constant multiple of a rational function in $z$, namely
\begin{align}\label{eqn_Akt4} A_k(z;w_{M_r^{(4)}}(z) ) = (4z+2)^{-k} \cdot \begin{cases} \kappa_{1,k}, & \text{if } r=1, \\
\kappa_{-1,k}, & \text{if } r=0,\\
i^{2k}\kappa_{-1,k}, & \text{if } r=-1, \\
e^{-\pi i k} \kappa_{1,k}, & \text{if } r=2.
 \end{cases}\end{align}

\section{Bounds} \label{secerrorbound}

In this section, we prove that the quantity \[\mathcal E_{m,k}(\theta) + e^{\frac{ik\theta}{2}} e^{-\frac{\pi m}{2} \sin(\theta)}\int_{-\frac12 + iv}^{\frac12 + iv} B_k(z;\tau) e(-m\tau) d\tau\] is bounded in absolute value by $2$, which we use in the proof of Theorem~\ref{mainthm}.  Recall that $\mathcal E_{m,k}(\theta)$ is equal to $\mathcal C_{m,k}(\theta)$ if $\theta \in \left(\frac{\pi}{3},\frac{5\pi}{12}\right]$, is equal to $\mathcal D_{m,k}(\theta)$ if $\theta \in \left[\frac{7\pi}{12}, \frac{2\pi}{3}\right)$, and is equal to zero otherwise, and that the functions
\begin{align*} \mathcal C_{m,k}(\theta) & = - c_{m,k}(2i)^{-k}\left(\sin(\tfrac{\theta}{2})\right)^{-k} e^{-\frac{\pi i m}{4}}
e^{\frac{\pi m}{2} \left(\frac{1}{2\tan\left(\frac{\theta}{2}\right)} - \sin(\theta)\right)}, \\
\mathcal D_{m,k}(\theta) & = -  d_{m,k}\cdot 2^{-k} \left(\cos(\tfrac{\theta}{2})\right)^{-k}   e^{\frac{\pi i m}{4}} e^{\frac{\pi m}{2}\left( \frac{\tan\left(\frac{\theta}{2}\right)}{2} -\sin(\theta)\right) } \end{align*}
were defined in equations~(\ref{def_Cmk}) and~(\ref{def_Dmk}).
We begin with the following lemma.

\begin{lemma}\label{cdboundlem}  Let $k = 12a + b + \frac{1}{2}$ as above.
If $m\geq 4.8|a|$, then the following are true.  \medskip \\ i) If $\theta \in \left(\frac{\pi}{3},\frac{5\pi}{12}\right]$, then $ |\mathcal C_{m,k}(\theta)|   < \sqrt{2}.$  \medskip \\ ii) If $\mathcal C_{m,k}(\theta)$, then $|\mathcal D_{m,k}(\theta) |  < \sqrt{2}.$
\end{lemma}
\begin{proof}
We first observe by definition that $|c_{m,k}e^{-\frac{\pi i m}{4}}| \leq \sqrt{2}$ and $|d_{m,k}e^{\frac{\pi i m}{4}}| \leq \sqrt{2}$.  We rewrite $m=\alpha|a| + \beta,$ where $\beta\geq 0$ and $\alpha \geq 4.8$, and $k=12a+b+\frac12$, where $a\in \mathbb Z$ and $b\in S$.  Let $\theta_1 \in \left(\frac{\pi}{3},\frac{5\pi}{12}\right]$ and $\theta_2 \in \left[\frac{7\pi}{12},\frac{2\pi}{3}\right)$.  For ease of notation, we define
\begin{align*} F(g,n;\theta)  := \left(2g\left(\tfrac{\theta}{2}\right)\right)^{-n}, \ \
G(g,n;\theta)  := e^{\frac{n\pi}{2}\left(\frac{1}{2}g\left(\frac{\theta}{2}\right) - \sin(\theta)\right)}.
\end{align*}

Then we have that \begin{align*} |\mathcal C_{m,k}(\theta_1)| & \leq \sqrt{2}  \left( F\left(\sin,b+\tfrac12;\theta_1\right) G\left(\frac{1}{\tan}, \beta;\theta_1\right)\right) \left(F(\sin,12a;\theta_1) G\left(\frac{1}{\tan}, \alpha|a|;\theta_1\right) \right),\\
|\mathcal D_{m,k}(\theta_2)| & \leq \sqrt{2}  \left( F\left(\cos,b+\tfrac12;\theta_2\right) G\left(\tan, \beta;\theta_2\right)\right) \left(F(\cos,12a;\theta_2) G\left(\tan, \alpha|a|;\theta_2\right) \right).
 \end{align*}
 The functions $\frac12\sin(\frac{\theta_1}{2})$ and $\cos(\frac{\theta_2}{2})$ are bounded between $1$ and $1.21752$, and the functions
$ \frac{1}{2\tan\left(\frac{\theta_1}{2}\right)} - \sin(\theta_1)$ and $\frac12\tan(\tfrac{\theta_2}{2}) - \sin(\theta_2)$ are bounded between $-.314313$ and $0$,
 when $\theta_1$ and $\theta_2$ are in the restricted domains given above.    Thus, since $\beta\geq 0$ and $\frac{13}{2}<b+\frac12 < \frac{39}{2}$, we have that
 \[ 0<F(\sin,b+\tfrac12;\theta_1) G\left(\frac{1}{\tan}, \beta;\theta_1\right)<1, \ \  0<F(\cos,b+\tfrac12;\theta_2) G\left(\tan, \beta;\theta_2\right) < 1,\] and hence \begin{align} |\mathcal C_{m,k}(\theta_1)| & < \sqrt{2}   F(\sin,12a;\theta_1) G\left(\frac{1}{\tan}, \alpha|a|;\theta_1\right), \label{cbd}\\
|\mathcal D_{m,k}(\theta_2)| & <\sqrt{2}   F(\cos,12a;\theta_2) G\left(\tan, \alpha|a|;\theta_2\right) \label{dbd}.
 \end{align}
 We consider two cases, $a\geq 0$ and $a<0$, and begin with the former.  By an identical argument given above, noting that $\alpha> 0$, we again have for $a\geq 0$ that \[0<F(\sin,12a;\theta_1) G\left(\frac{1}{\tan}, \alpha|a|;\theta_1\right)\leq 1, \ \ 0 < F(\cos,12a;\theta_2) G\left(\tan, \alpha|a|;\theta_2\right) \leq 1,\] and hence, using this together with (\ref{cbd}) and (\ref{dbd}), for $a\geq 0$, Lemma~\ref{cdboundlem} is proved.  In the second case, where $a<0$, we rewrite $-12a=12|a|$, so that
\begin{align}\label{aneg}F(g,12a;\theta) G\left(h, \alpha|a|;\theta\right) = \left(\left(2g\left(\tfrac{\theta}{2}\right)\right)^{12} e^{\frac{\alpha\pi}{2}\left(\frac{1}{2}h(\frac{\theta}{2}) - \sin(\theta)\right)}\right)^{|a|}.\end{align}  If $\alpha\geq 4.8$, then the functions
 \begin{align*} & \left(2\sin\left(\tfrac{\theta_1}{2}\right)\right)^{12} e^{\frac{\alpha\pi}{2}\left(\frac{1}{2\tan\left(\frac{\theta_1}{2}\right)} - \sin(\theta_1)\right)},  \\
 & \left(2\cos\left(\tfrac{\theta_2}{2}\right)\right)^{12} e^{\frac{\alpha\pi}{2}\left(\frac{\tan\left(\frac{\theta_2}{2}\right)}{2} - \sin(\theta_2)\right)},\end{align*} are bounded between $0$ and $1$ when $\theta_1$ and $\theta_2$ are restricted to the domains given above.  Thus, since $|a|>0$, using this with (\ref{cbd}),  (\ref{dbd}), and (\ref{aneg}), Lemma~\ref{cdboundlem} is also proved in the second case, $a<0$.
\end{proof}

We next bound the integral
\[\int_{-\frac{1}{2} + iv}^{\frac{1}{2} + i v} \frac{f_k(z)f_{2-k}^*(\tau) + f_k^*(z)f_{2-k}(\tau)}{j(4\tau)-j(4z)}e^{-2\pi i m \tau} d\tau,\] where $\tau = u+iv$ and $v$ is fixed.  We rewrite the integrand as
\[ \left(\frac{\Delta(4z)}{\Delta(4\tau)}\right)^a \frac{f_b(z)f_{2-b}^*(\tau) + f_b^*(z)f_{2-b}(\tau)}{j(4\tau)-j(4z)}e^{-2\pi i m \tau}.\]
To bound this, we replace the integral with the absolute value of the integrand, since the contour has length 1.  Since $a$ can be either positive or negative, we need upper and lower bounds on $\Delta$ for the appropriate values of $\tau$ and $z$.  Additionally, we need a lower bound on the absolute value of the denominator $j(4\tau) - j(4z)$.  Finally, we need upper bounds on the basis elements $f_b$ and $f_b^*$, which can be obtained (in the absence of explicit bounds on their coefficients)  by writing them in terms of the Eisenstein series $F(\tau)$ and the theta function $\vartheta(\tau)$, as detailed in the appendix of~\cite{DJ2}.

If $\frac{5\pi}{12} \leq \theta \leq \frac{7\pi}{12}$, we write $z_1 = -\frac{1}{4} + \frac{1}{4} e^{i\theta}$, and let $v_1 = 0.2125$.
If $\frac{\pi}{3} < \theta < \frac{5\pi}{12}$ or $\frac{7\pi}{12} < \theta < \frac{2\pi}{3}$, we instead write $z_2 = -\frac{1}{4} + \frac{1}{4} e^{i\theta}$ and let $v_2 = 0.1375$.  We have computed the following bounds using MAPLE for these values of $\tau$ and $z$.  Generally, this requires explicitly bounding the tail of the Fourier expansion using crude bounds on the Fourier coefficients, and then computing bounds on the initial terms on the appropriate interval.  In some cases, these bounds on the initial terms require bounding the derivative and checking values on a grid of points, similar to the computations in~\cite{GJ}.

For the theta function $\vartheta(\tau) = 1 + 2q + 2q^4 + \ldots$, we have
\[\abs{\vartheta(\tau_1)} \leq 1.53583,\]
\[\abs{\vartheta(\tau_2)} \leq 1.90697,\]
\[\abs{\vartheta(z_1)} \leq 1.44325,\]
\[\abs{\vartheta(z_2)} \leq 1.52182.\]
For the delta function $\Delta(\tau) = q - 24q + \ldots$, we have
\[0.00407 \leq \abs{\Delta(4\tau_1)} \leq 0.00551,\]
\[0.01 \leq \abs{\Delta(4\tau_2)} \leq 0.11054,\]
\[0.0015 \leq \abs{\Delta(4z_1)} \leq 0.00246,\]
\[0.0015 \leq \abs{\Delta(4z_2)} \leq 0.00491.\]
For the weight 2, level 4 Eisenstein series $F(\tau)$, we have
\[\abs{F(\tau_1)} \leq 0.34440,\]
\[\abs{F(\tau_2)} \leq 0.82688,\]
\[\abs{F(z_1)} \leq 0.26477,\]
\[\abs{F(z_2)} \leq 0.33151.\]
For the $j$-function, we know that $j(4z)$ is real, with $582.84 \leq j(4z_1) \leq 1728$ and $0 \leq j(4z_2) \leq 582.9$, and we have
\[\abs{j(4\tau_1) - j(4z_1)} \geq 132,\]
\[\abs{j(4\tau_2) - j(4z_2)} \geq 1200.\]

Suppose that $k = 12a + 6 +\frac{1}{2}$ for some integer $a$.  The integral is bounded above by
\begin{align}\label{eqn_intbd1} \left|\frac{\Delta(4z)}{\Delta(4\tau)}\right|^a \frac{\abs{f_{\frac{13}{2}}(z)f_{\frac{-9}{2}}^*(\tau) + f_{\frac{13}{2}}^*(z)f_{\frac{-9}{2}}(\tau)}}{\abs{j(4\tau)-j(4z)}} e^{2\pi m v}.\end{align}  Note that
\[f_{\frac{-9}{2}}(\tau) = \frac{f_{\frac{39}{2}}(\tau)}{\Delta(4\tau)^2}, \ \ \textnormal{and } \ \
f_{\frac{-9}{2}}^*(\tau) = \frac{f_{\frac{39}{2}}^*(\tau)}{\Delta(4\tau)^2}.\]  Thus, (\ref{eqn_intbd1}) can be rewritten as
\[ \left|\frac{\Delta(4z)}{\Delta(4\tau)}\right|^a \frac{\abs{f_{\frac{13}{2}}(z)f_{\frac{39}{2}}^*(\tau) + f_{\frac{13}{2}}^*(z)f_{\frac{39}{2}}(\tau)}}{\abs{\Delta(4\tau)^2}\abs{j(4\tau)-j(4z)}} e^{2\pi m v}.\]

For $z_1$ and $\tau_1$, we have
\[\abs{f_{\frac{13}{2}}(z_1)} \leq 15.95180790,\]
\[\abs{f_{\frac{13}{2}}^*(z_1)} \leq 373.3811270,\]
\[\abs{f_{\frac{39}{2}}(\tau_1)} \leq 2070.877536,\]
\[\abs{f_{\frac{39}{2}}^*(\tau_1)} \leq 48384.64244.\]
This gives a bound on the integral of
\[(0.60443)^a \cdot 706609609 \cdot e^{\frac{.85\pi m}{2}}\] if $a \geq 0$, and
\[(3.6734)^{-a} \cdot 706609609 \cdot e^{\frac{.85\pi m}{2}}\] if $a < 0$.
Thus, we have that
\[e^{\frac{ik\theta}{2}} e^{-\frac{\pi m}{2}\sin(\theta)} f_{k, m}(z) - 2\cos\left(\frac{k\theta}{2} + \frac{\pi m}{2} - \frac{\pi m \cos(\theta)}{2}\right) \leq e^{-\frac{\pi m}{2} (\sin(\theta) - .85)} (0.60443)^a \cdot 706609609\]
if $a \geq 0$, and
\[e^{\frac{ik\theta}{2}} e^{-\frac{\pi m}{2}\sin(\theta)} f_{k, m}(z) - 2\cos\left(\frac{k\theta}{2} + \frac{\pi m}{2} - \frac{\pi m \cos(\theta)}{2}\right) \leq e^{-\frac{\pi m}{2} (\sin(\theta) - .85)} (3.6734)^{-a} \cdot 706609609\]
if $a < 0$.  The  two quantities on the right above can  be rewritten as
\[ \leq (0.83353)^m (0.60443)^a \cdot 706609609,\]
\[ \leq (0.83353)^m (3.6734)^{-a} \cdot 706609609.\]
We have $(0.83353)^m \cdot 706609609 \leq 2$ for $m \geq 109$, so the quantity for $a \geq 0$ is less than $2$ for $m \geq 109$.
If $a < 0$, we have $(0.83353)^m \cdot 3.6734 \leq 1$ for $m \geq 8$, so the quantity is less than $2$ for $m \geq 8\abs{a}+109$.

For $z_2$ and $\tau_2$, we have
\[\abs{f_{\frac{13}{2}}(z_2)} \leq 32.43415,\]
\[\abs{f_{\frac{13}{2}}^*(z_2)} \leq 752.09673,\]
\[\abs{f_{\frac{39}{2}}(\tau_2)} \leq 10147561.12,\]
\[\abs{f_{\frac{39}{2}}^*(\tau_2)} \leq 235391937.7.\]
This gives a bound on the integral of
\[ (0.491)^a \cdot 127222365876 \cdot e^{\frac{.55\pi m}{2}}\] if $a \geq 0$, and
\[ (73.6934)^{-a} \cdot 127222365876 \cdot e^{\frac{.55\pi m}{2}}\] if $a < 0$.

Using the bounds on $\mathcal{C}_{m, k}(\theta_2)$ and $\mathcal{D}_{m, k}(\theta_2)$ from Lemma \ref{cdboundlem} above, we find that the difference between the weighted modular form and the cosine function is bounded above by
\[ \sqrt{2} + e^{-\frac{\pi m}{2} (\sin(\theta) - .55)} (0.491)^a \cdot 127222365876\] if $a \geq 0$ and, if $a < 0$, by
\[ \sqrt{2} + e^{-\frac{\pi m}{2} (\sin(\theta) - .55)} (73.6934)^{-a} \cdot 127222365876.\]  These simplify to
\[ \leq \sqrt{2} + (0.60872)^m (0.491)^a \cdot 127222365876,\]
\[ \leq \sqrt{2} + (0.60872)^m (73.6934)^{-a} \cdot 127222365876,\]
giving that for $a \geq 0$, the quantity is less than $2$ if $m \geq 53$, and for $a < 0$, it is less than $2$ if $m \geq 9|a| + 53$.

For the other eleven values of $b$, similar computations show that the bounds on \[\abs{f_{b}(z)f_{2-b}^*(\tau) + f_{b}^*(z)f_{2-b}(\tau)}\] become smaller.  Specifically, for each value of $b$, we find that the numbers 706609609 and 127222365876 in cases $1$ and $2$ may be replaced by the values in Table~\ref{boundtable}.
Thus, in all cases the quantity \[\mathcal E_{m,k}(\theta) + e^{\frac{ik\theta}{2}} e^{-\frac{\pi m}{2} \sin(\theta)}\int_{-\frac12 + iv}^{\frac12 + iv} B_k(z;\tau) e(-m\tau) d\tau\] is bounded in absolute value by $2$.  \begin{table}[!ht]
  \centering
  \caption{Upper bounds for each $b$}\label{boundtable}
    \begin{tabular}{|r|r|r|}
    \hline
    $b$ & For $z_1, \tau_1$ & For $z_2, \tau_2$ \\
    \hline
    6 & 706609608 & 127222365875 \\
    8 & 554055912 & 51930014336 \\
    9 & 478100088 & 32392878212 \\
    10 & 427574714 & 20854624833 \\
    11 & 408921890 & 14417163525 \\
    12 & 325238946 & 8284899739 \\
    13 & 288599577 & 5296681421 \\
    14 & 273853210 & 3640432156 \\
    15 & 220558615 & 2114574952 \\
    16 & 196218970 & 1353920641 \\
    17 & 172470466 & 860720673 \\
    19 & 132750791 & 344722508 \\
    \hline
    \end{tabular}
\end{table}

\bibliographystyle{amsplain}
%\bibliography{ZerosBibliography}

\begin{thebibliography}{10}

\bibitem{BF}
J.~Bruinier and J.~Funke, \emph{Traces of {C}{M}-values of modular functions},
  J. Reine Angew. Math., accepted for publication.

\bibitem{BJO}
J.~Bruinier, P.~Jenkins, and K.~Ono, \emph{Hilbert class polynomials and traces
  of singular moduli}, Math. Ann. \textbf{334} (2006), no.~2, 373--393.

\bibitem{ChenWu}
B.~Chen and J.~Wu, \emph{Non-vanishing and sign changes of {H}ecke eigenvalues
  for half-integral weight cusp forms}, arXiv:1512.08400v1 [math.NT].

\bibitem{DJ2}
W.~Duke and P.~Jenkins, \emph{Integral traces of singular values of weak
  {M}aass forms}, Algebra and Number Theory \textbf{2} (2008), no.~5, 573--593.

\bibitem{DJ1}
\bysame, \emph{On the zeros and coefficients of certain weakly holomorphic
  modular forms}, Pure Appl. Math. Q. \textbf{4} (2008), no.~4, 1327--1340.

\bibitem{GJ}
S.~A. Garthwaite and P.~Jenkins, \emph{Zeros of weakly holomorphic modular
  forms of levels 2 and 3}, Math. Res. Lett. \textbf{20} (2013), no.~4,
  657--674.

\bibitem{GS}
A.~Ghosh and P.~Sarnak, \emph{Real zeros of holomorphic {H}ecke cusp forms},
  Jour. Eur. Math. Soc. \textbf{14} (2012), no.~2, 465--487.

\bibitem{GreenJenkins}
N.~Green and P.~Jenkins, \emph{Integral traces of weak maass forms of genus
  zero odd prime level}, The Ramanujan Journal, to appear, arXiv:1370.2204v1
  [math.NT].

\bibitem{HJ}
A.~Haddock and P.~Jenkins, \emph{Zeros of weakly holomorphic modular forms of
  level 4}, Int. J. Number Theory \textbf{10} (2014), no.~2, 455--470.

\bibitem{HS}
R.~Holowinsky and K.~Soundararajan, \emph{Mass equidistribution for {H}ecke
  eigenforms}, Ann. of Math. (2) \textbf{172} (2010), no.~2, 1517--1528.

\bibitem{Koblitz}
N.~Koblitz, \emph{Introduction to elliptic curves and modular forms}, Graduate
  Texts in Mathematics, vol.~97, Springer-Verlag, New York, 1984.

\bibitem{Kohnen}
W.~Kohnen, \emph{Fourier coefficients of modular forms of half integral
  weight}, Math. Ann. \textbf{271} (1985), 237--268.

\bibitem{LMR}
S.~Lester, K.~Matom\"aki, and M.~Radziwill, \emph{Zeros of modular forms in
  thin sets and effective quantum unique ergodicity}, Preprint.

\bibitem{OnoCDM}
K.~Ono, \emph{Unearthing the visions of a master: harmonic {M}aass forms and
  number theory}, Current developments in mathematics, 2008, Int. Press,
  Somerville, MA, 2009, pp.~347--454. \MR{2555930 (2010m:11060)}

\bibitem{Rankin}
R.~A. Rankin, \emph{Modular forms and functions}, Cambridge University Press,
  Cambridge-New York-Melbourne, 1977.

\bibitem{Ra1}
\bysame, \emph{The zeros of certain {P}oincar\'e series}, Compositio Math.
  \textbf{46} (1982), no.~3, 255--272.

\bibitem{Ru}
Z.~Rudnick, \emph{On the asymptotic distribution of zeros of modular forms},
  Int. Math. Res. Not. (2005), no.~34, 2059--2074.

\bibitem{Sturm}
J.~Sturm, \emph{On the congruence of modular forms}, Number theory ({N}ew
  {Y}ork, 1984--1985), Lecture Notes in Math., vol. 1240, Springer, Berlin,
  1987, pp.~275--280.

\bibitem{ZagierEis}
D.~Zagier, \emph{Nombres de classes et formes modulaires de poids {$3/2$}}, C.
  R. Acad. Sci. Paris S\'er. A-B \textbf{281} (1975), no.~21, Ai, A883--A886.
  \MR{0429750 (55 \#2760)}

\bibitem{Z}
\bysame, \emph{Traces of singular moduli}, Motives, polylogarithms and {H}odge
  theory, Part I (Irvine, CA, 1998), Int. Press Lect. Ser., vol.~3, Int. Press,
  Somerville, MA, 2002, pp.~211--244.

\end{thebibliography}

\end{document}